\documentclass[12pt]{article}
\usepackage[top=30truemm,bottom=30truemm,left=20truemm,right=20truemm]{geometry}
\usepackage{amsmath}
\usepackage{amsthm}
\usepackage{amssymb}
\usepackage{amsfonts}
\usepackage{mathrsfs}
\usepackage{mathtools}
\usepackage{cite}
\usepackage{comment}
\usepackage{enumitem}
\usepackage{lineno}

\numberwithin{equation}{section}

\newtheorem*{ass*}{\textbf{Assumption}}
\newtheorem{thm}{\textbf{Theorem}}[section]
\newtheorem*{thm*}{\textbf{Theorem}}
\newtheorem{prop}[thm]{\textbf{Proposition}}
\newtheorem*{prop*}{\textbf{Proposition}}
\newtheorem{lem}[thm]{\textbf{Lemma}}
\newtheorem*{lem*}{\textbf{Lemma}}

\theoremstyle{definition}

\newtheorem*{dfn*}{\textbf{Definition}}
\newtheorem{rem}[thm]{\textbf{Remark}}
\newtheorem*{rem*}{\textbf{Remark}}

\mathtoolsset{showonlyrefs}

\title{Stability of travelling waves to Korteweg--de Vries type equations with fractional dispersion}
\author{Kaito KOKUBU\thanks{Department of Mathematics, Graduate School of Science, Tokyo University of Science. 1-3 Kagurazaka, Shinjuku-ku, Tokyo 162-8601, Japan.
e-mail: \texttt{1123701@ed.tus.ac.jp}}}
\date{}

\begin{document}

\maketitle

\begin{abstract}
    We study stability of travelling wave solutions to Korteweg--de Vries type equations which has the fractional dispersion and integer-indices double power nonlinearities.
    It may depend on parity combinations of the two indices and the strength of dispersion whether these equations have a ground state solution.
    Therefore, we observe the stability phenomena on travelling wave solutions from the perspective of the parities and the dispersion, and we give the classification of phenomena on travelling wave solutions.
    In this paper, we focus on stable travelling wave solutions.
\end{abstract}


\section{Introduction} \label{section:Introduction;gkdv_st}
In this paper, we consider the stability of travelling wave solutions to Korteweg--de Vries type equations
\begin{equation}
    \partial_{t} u + \partial_{x} f(u) - \partial_{x} D_{x}^{\sigma} u = 0, \quad (t,x) \in \mathbb{R} \times \mathbb{R}, \label{eq:Intro0100;gkdv_st}
\end{equation}
where $u=u(t,x)$ is a real-valued unknown function, $\sigma$ is a real number satisfying $1 \leq \sigma \leq 2$, and the operator $D_{x}^{\sigma}$ is the Fourier multiplier with symbol $|\xi|^{\sigma}$.
The operator $D_{x}^{\sigma}$ is also denoted as $(- \partial_{x}^{2})^{\sigma/2}$.

When $\sigma = 2$ and $f(s) = s^{2} (s \in \mathbb{R})$, the equation \eqref{eq:Intro0100;gkdv_st} coincides with the Korteweg--de Vries equation, which describes physical dynamics of waves on shallow water (see, e.g.\ \cite{Korteweg-deVries}).
When $\sigma = 1$ and $f(s) = s^{2}$, the equation \eqref{eq:Intro0100;gkdv_st} is the Benjamin--Ono equation, which physically describes dynamics of internal waves in stratified fluids (see, e.g.\ \cite{Benjamin,Ono_1975}).

A travelling wave solution is a solution to \eqref{eq:Intro0100;gkdv_st} of form $u(t,x) = \phi(x-ct)$, where $c$ is a positive constant representing the speed of the wave, and $\phi \in H^{\sigma/2}(\mathbb{R})$ is a nontrivial solution to the stationary problem
\begin{equation}
    D_{x}^{\sigma} \phi + c \phi - f(\phi) = 0, \quad x \in \mathbb{R}. \label{eq:Intro0400;gkdv_st}
\end{equation}
We say that $\phi \in H^{\sigma/2}(\mathbb{R})$ is \textit{a ground state solution} to \eqref{eq:Intro0400;gkdv_st} if $\phi$ satisfies
\begin{equation}
    S_{c}(\phi) = \inf \{ S_{c}(v): v \in H^{\sigma/2}(\mathbb{R}) \setminus \{0\}, \ v \ \text{is a solution to \eqref{eq:Intro0400;gkdv_st}} \},
\end{equation}
where the functional $S_{c}$ is the action functional corresponding to the equation \eqref{eq:Intro0400;gkdv_st} defined in $H^{\sigma/2}(\mathbb{R})$ as $S_{c}(v) = E(v) + c M(v)$ with the energy functional $E$ and the mass functional $M$ defined as
\begin{equation}
    E(v) \coloneq \frac{1}{2} \| D_{x}^{\sigma/2} u \|_{L^{2}}^{2} - \int_{\mathbb{R}} F(u) \, dx, \quad M(v) \coloneq \frac{1}{2} \| v \|_{L^{2}}^{2},
\end{equation}
where $F(s)$ is the primitive function of the nonlinearity $f(s)$.
We remark that $S_{c} \in C^{2}(H^{\sigma/2}(\mathbb{R}), \mathbb{R})$\footnote{In this paper, we mainly consider power type nonlinearities, which are good enough for the functional $S_{c}$ to belong to $C^{2}(H^{\sigma/2}(\mathbb{R}), \mathbb{R})$.}
and that $v \in H^{\sigma/2}(\mathbb{R})$ is a solution to \eqref{eq:Intro0400;gkdv_st} if and only if $S'(v) = 0$, where $G'$ denotes the Fr\'{e}chet derivative of a functional $G$ defined in $H^{\sigma/2}(\mathbb{R})$.
In this paper, we mainly consider travelling wave solutions to \eqref{eq:Intro0400;gkdv_st} constructed by ground state solutions to \eqref{eq:Intro0400;gkdv_st}.

Next, we define the stability and instability of travelling wave solutions.
First, for $r > 0$ and $v \in H^{\sigma/2}(\mathbb{R})$, we set
\begin{equation}
    U_{r}(v) \coloneq \{ u \in H^{\sigma/2}(\mathbb{R}): \inf_{y \in \mathbb{R}} \| v - u(\cdot - y) \|_{H^{\sigma/2}}^{2} < r \}.
\end{equation}
Let $\phi \in H^{\sigma/2}(\mathbb{R})$ be a nontrivial solution to \eqref{eq:Intro0400;gkdv_st}.
We say that a travelling wave solution $\phi(x-ct)$ to \eqref{eq:Intro0100;gkdv_st} is \textit{stable} if for any $\varepsilon > 0$, there exists $\delta > 0$ such that if $u_{0} \in U_{\varepsilon}(\phi)$, then the time-global solution $u(t) \in C([0,\infty), H^{\sigma/2}(\mathbb{R}))$ to \eqref{eq:Intro0100;gkdv_st} exists and satisfies that $u(t) \in U_{\varepsilon}(\phi)$ holds for all $t \geq 0$.
Otherwise, we say that a travelling wave solution $\phi(x-ct)$ is \textit{unstable}.

The stability and instability of travelling wave solutions to \eqref{eq:Intro0100;gkdv_st} with single power nonlinearities $f(s) = s^{p}$ ($p \in \mathbb{N}$, $2 \leq p < \infty$) have been studied well.
In this case, it is known that the stationary problem \eqref{eq:Intro0400;gkdv_st} has the unique positive and even ground state solution (existence: Weinstein~\cite{Weinstein}, uniqueness for $1 \leq \sigma < 2$: Frank--Lenzmann~\cite{Frank-Lenzmann}).
By an abstract theory of the stability and instability of travelling wave solutions developed by Bona--Souganidis--Strauss~\cite{BSS1987}, we can find that the travelling wave solution to \eqref{eq:Intro0100;gkdv_st} constructed with the positive ground state solution to \eqref{eq:Intro0400;gkdv_st} is stable if $p < 2\sigma + 1$, or unstable if $p > 2\sigma + 1$
(for the stability results, see also \cite{Weinstein}).
According to \cite{BSS1987}, travelling wave solutions are stable if $\partial_{c}^{2} d(c) > 0$, or unstable if $\partial_{c}^{2} d(c) < 0$, where $d(c) = S_{c}(\psi_{c})$ with $\psi_{c}$ denoting a ground state solution for $c > 0$.
When nonlinearities of \eqref{eq:Intro0100;gkdv_st} are single power ones, we can find the scaling property of the ground state solution to \eqref{eq:Intro0400;gkdv_st} such that $\psi_{c}(x) = c^{1/(p-1)} \psi_{1}(c^{1/\sigma} x)$.
This property allows us to calculate $\partial_{c}^{2} d(c)$ easily so that we can see that $\partial_{c}^{2} d(c) > 0$ holds if $p < 2\sigma + 1$, or $\partial_{c}^{2} d(c) < 0$ if $p > 2\sigma + 1$.
Here we remark that the exponent $2 \sigma + 1$ is the $L^{2}$-critical one for the KdV type equation \eqref{eq:Intro0100;gkdv_st}.
When $\sigma = 2$ and $p = 2 \sigma + 1 = 5$, Martel--Merle~\cite{Martel-Merle_2001} proved that the travelling wave solution to \eqref{eq:Intro0100;gkdv_st} is unstable.

Meanwhile, it is difficult to analyze the stability of travelling wave solutions to the equation \eqref{eq:Intro0100;gkdv_st} with generalized nonlinearities because we can hardly see the signature of $\partial_{c}^{2} d(c)$ so that the theory of \cite{BSS1987} is not applicable.
Therefore, the purpose of this study is to focus on the equation \eqref{eq:Intro0100;gkdv_st} with specific nonlinearities and observe stability properties of travelling wave solutions.
In this paper, we consider equations \eqref{eq:Intro0100;gkdv_st} with integer-indices double power nonlinearities such as
\begin{equation}
    \partial_{t} u + \partial_{x} (a u^{p} + u^{q}) - \partial_{x} D_{x}^{\sigma} u = 0, \quad (t,x) \in \mathbb{R} \times \mathbb{R}, \label{eq:Intro0200;gkdv_st}
\end{equation}
where $a \in \{ +1, -1 \}$ and $p, q \in \mathbb{N}$ satisfy $2 \leq p < q < \infty$.
The remarkable example of these equations is the case that $\sigma = 2$, $p = 2$, and $q = 3$, where the equation \eqref{eq:Intro0200;gkdv_st} coincides with the so-called Gardner equation, which is introduced by Miura--Gardner--Kruskal~\cite{Miura-Gardner-Kruskal_II}. 

Throughout this paper, we assume the local well-posedness of the Cauchy problem associated with \eqref{eq:Intro0200;gkdv_st} in the energy space $H^{\sigma/2}(\mathbb{R})$.
\begin{ass*} \label{ass:Wellposedness;gkdv_st}
    Let $1 \leq \sigma \leq 2$.
    Then, for any $u_{0} \in H^{\sigma/2}(\mathbb{R})$, there exists $T > 0$ and a unique solution $u(t) \in C([0,T), H^{\sigma/2}(\mathbb{R}))$ to \eqref{eq:Intro0200;gkdv_st} with $u(0) = u_{0}$ which satisfy the following conservation laws:
    \begin{equation}
        E(u(t)) = E(u_{0}), \quad M(u(t)) = M(u_{0}), \quad t \in [0,T).
    \end{equation}
\end{ass*}
We remark that Molinet--Tanaka~\cite{Molinet-Tanaka_2024} showed that this assumption actually holds if $4/3 \leq \sigma \leq 2$.

The stationary problem derived from \eqref{eq:Intro0200;gkdv_st} is the following:
\begin{equation}
    D_{x}^{\sigma} \phi + c \phi - a \phi^{p} - \phi^{q} = 0, \quad x \in \mathbb{R}. \label{eq:Intro0300;gkdv_st}
\end{equation}
The author studied the existence of ground state solutions to \eqref{eq:Intro0200;gkdv_st} with $a = -1$ in a previous paper~\cite{Kokubu2024} and found that existence properties of nontrivial solutions to \eqref{eq:Intro0400;gkdv_st} depend on parity combinations of indices $p$ and $q$.
Here we state the results of the existence of ground state solutions to \eqref{eq:Intro0300;gkdv_st} including the case $a = +1$.

\begin{thm} \label{thm:existence_of_GS;gkdv_st}
    Let $1 \leq \sigma \leq 2$, $p, q \in \mathbb{N}$, $2 \leq p < q < \infty$, and $c > 0$.
    Then the following properties hold:
    \begin{enumerate}[label={\rm (\Roman*)} \ ]
        \item The case $a = +1$.

        If $q$ is odd, then there exists a positive ground state solution to \eqref{eq:Intro0300;gkdv_st}.

        \item The case $a = -1$.

            \begin{enumerate}[label={\rm (\arabic*)} \ ]
                \item If $p$ is odd, then there exists a positive ground state solution to \eqref{eq:Intro0300;gkdv_st}.

                \item If $p$ is even and $q$ is odd, then there exists a negative ground state solution to \eqref{eq:Intro0300;gkdv_st}.
            \end{enumerate}
    \end{enumerate}
\end{thm}

\begin{rem}
    \begin{enumerate}[label=\arabic*.]
        \item Each ground state solution obtained in Theorem \ref{thm:existence_of_GS;gkdv_st} can be taken as an even function, and be decreasing in $|x|$ if positive or increasing in $|x|$ if negative.
        Moreover, since $p$ and $q$ are positive integers, we can see that any ground state solution to \eqref{eq:Intro0300;gkdv_st}, which is found in $H^{\sigma/2}(\mathbb{R})$, belongs to $H^{\infty}(\mathbb{R})$.
        We will prove this in Section \ref{section:existence_of_GS;gkdv_st}.

        \item In Theorem \ref{thm:existence_of_GS;gkdv_st}, we only consider cases that we can find a ground state solution, while there are other cases not appearing in the statement where we can find a nontrivial solution which is (possibly) not ground state solutions.
        In addition, when we consider case (II-2), we can show that there exists a positive solution to \eqref{eq:Intro0400;gkdv_st} and none of them are ground state solutions.
        For details, see the author's previous paper \cite{Kokubu2024}.
    \end{enumerate}
    
\end{rem}

Hereafter, whenever we take a ground state solution to \eqref{eq:Intro0300;gkdv_st} or any stationary problem appearing below, we always consider even one.

Now we state the stability result of travelling wave solutions to \eqref{eq:Intro0200;gkdv_st} constructed with ground state solutions to \eqref{eq:Intro0300;gkdv_st} obtained in Theorem \ref{thm:existence_of_GS;gkdv_st}.

\begin{thm} \label{thm:stability;gkdv_st}
    Let $1 \leq \sigma \leq 2$, $p, q \in \mathbb{N}$, and $2 \leq p < q < \infty$.
    Then the following properties hold:
    \begin{enumerate}[label={\rm (\Roman*)} \ ]
        \item The case $a = +1$.

        Assume that $q$ is odd, and $p < 2 \sigma + 1$.
        Let $\phi_{c}$ be a positive ground state to \eqref{eq:Intro0300;gkdv_st} for $c > 0$.
        Then there exists $c_{0} \in (0,\infty)$ such that a travelling wave solution $\phi_{c}(x-ct)$ to \eqref{eq:Intro0200;gkdv_st} is stable for all $c \in (0, c_{0})$.

        \item The case $a = -1$.

        \begin{enumerate}[label={\rm (\arabic*)} \ ]
            \item Assume that $p$ is odd and $q < 2 \sigma + 1$.
            Let $\phi_{c}$ be a positive ground state solution to \eqref{eq:Intro0300;gkdv_st} for $c > 0$.
            Then there exists $c_{1} \in (0, \infty)$ such that a travelling wave solution $\phi_{c}(x-ct)$ to \eqref{eq:Intro0200;gkdv_st} is stable for all $c \in (c_{1}, \infty)$.
            
            \item Assume that $p$ is even, $q$ is odd, and $q < 2 \sigma + 1$.
            Let $\phi_{c}$ be a negative ground state solution to \eqref{eq:Intro0300;gkdv_st} for $c > 0$.
            Then there exists $c_{2} \in (0, \infty)$ such that a travelling wave solution $\phi_{c}(x-ct)$ to \eqref{eq:Intro0200;gkdv_st} is stable for all $c \in (c_{2}, \infty)$.
        \end{enumerate}
    \end{enumerate}
\end{thm}

\subsection*{Plan for the paper}
We prove Theorem \ref{thm:existence_of_GS;gkdv_st} in Section \ref{section:existence_of_GS;gkdv_st}.
In Section \ref{section:stability;gkdv_st}, we consider sufficient conditions for the stability of travelling wave solutions.
In section \ref{section:suf_cond_2;gkdv_st}, we prove that the sufficient conditions given in Section \ref{section:stability;gkdv_st} hold true.

\subsection*{Notation}
For $s > 0$ and $u, v \in H^{s}(\mathbb{R})$, we define an inner product as
\begin{equation}
    (u, v)_{H^{s}} \coloneq (D_{x}^{s}u, D_{x}^{s}v)_{L^{2}} + (u, v)_{L^{2}}.
\end{equation}
Let $\| \cdot \|_{H^{s}}$ denote the norm of $H^{s}(\mathbb{R})$ naturally defined by the inner product.
Additionally, we let $\langle f, v \rangle$ denote a dual product between $H^{-\sigma/2}(\mathbb{R})$ and $H^{\sigma/2}(\mathbb{R})$.

To simplify the notation, we often use the same letter for constants in different estimates such as $A, B$, and $C$.
Any subsequence appearing below will be denoted by original characters.


\section{Ground state solutions} \label{section:existence_of_GS;gkdv_st}
As mentioned in Section \ref{section:Introduction;gkdv_st}, the existence of ground state solutions to \eqref{eq:Intro0300;gkdv_st} for case (II-$j$) ($j = 1, 2$) has been studied by the author.
In this section, we will complete the proof of Theorem \ref{thm:existence_of_GS;gkdv_st}.

\subsection{Existence} \label{subsection:GS_existence;gkdv_st}
Here we observe the existence of ground state solutions to \eqref{eq:Intro0300;gkdv_st}.
In this subsection, we always assume one of the three conditions of (I), (II-1), or (II-2) in Theorem \ref{thm:existence_of_GS;gkdv_st}, and assume that $c > 0$.

We let $\mathcal{G}_{c}$ denote the set of ground state solutions to \eqref{eq:Intro0300;gkdv_st}.
The action functional $S_{c}$ corresponding to \eqref{eq:Intro0300;gkdv_st} is written as
\begin{equation}
    S_{c}(v) = \frac{1}{2} \| v \|_{H^{\sigma/2}_{c}}^{2} - \frac{a}{p+1} \int_{\mathbb{R}} v^{p+1} \, dx - \frac{1}{q+1} \int_{\mathbb{R}} v^{q+1} \, dx
\end{equation}
for $v \in H^{\sigma/2}(\mathbb{R})$, where $\| v \|_{H^{\sigma/2}_{c}}^{2} \coloneq \| D_{x}^{\sigma/2} v \|_{L^{2}}^{2} + c \| v \|_{L^{2}}^{2}$.
We define the Nehari functional $K_{c}$ derived from the action functional $S_{c}$ as
\begin{equation}
    K_{c}(v) \coloneq \langle S_{c}'(v), v \rangle = \| v \|_{H^{\sigma/2}_{c}}^{2} - a \int_{\mathbb{R}} v^{p+1} \, dx - \int_{\mathbb{R}} v^{q+1} \, dx.
\end{equation}
We remark that $K_{c}(v) = 0$ holds for any nontrivial solution to \eqref{eq:Intro0300;gkdv_st}.
Moreover, we set
\begin{align}
    & d(c) \coloneq \inf \{ S_{c}(v): v \in H^{\sigma/2}(\mathbb{R}) \setminus \{0\}, \ K_{c}(v)= 0 \}, \\
    & \mathcal{M}_{c} \coloneq \{ v \in H^{\sigma/2}(\mathbb{R}) \setminus \{0\}: K_{c}(v) = 0, \ S_{c}(v) = d(c) \}.
\end{align}

\begin{lem} \label{lem:existence_gs_GM;gkdv_st}
    If $\mathcal{M}_{c} \neq \emptyset$, then $\mathcal{G}_{c} = \mathcal{M}_{c}$ holds.
\end{lem}
\begin{proof}
    First, we show that $\mathcal{M}_{c} \subset \mathcal{G}_{c}$.
    Let $\phi \in \mathcal{M}_{c}$.
    Then we see that $K_{c}(\phi) = 0$.
    Now we consider the following function:
    \begin{equation}
        (0, \infty) \ni \lambda \mapsto K_{c}(\lambda \phi) = \lambda^{2} \| v \|_{H^{\sigma/2}_{c}}^{2} - a \lambda^{p+1} \int_{\mathbb{R}} \phi^{p+1} \, dx - \lambda^{q+1} \int_{\mathbb{R}} \phi^{q+1} \, dx. \label{eq:GS0100;gkdv_st}
    \end{equation}
    We show that
    \begin{equation}
        \left. \partial_{\lambda} K_{c}(\lambda \phi) \right|_{c = 1} = \langle K_{c}'(\phi), \phi \rangle < 0. \label{eq:GS0200_gkdv_st}
    \end{equation}
    If $q$ is odd, it is clear that $\int_{\mathbb{R}} \phi^{q+1} \, dx > 0$.
    Moreover, if $a = -1$ and $p$ is odd, we obtain it from $K_{c}(\phi) = 0$ that
    \begin{equation}
        0 < \| \phi \|_{H^{\sigma/2}_{c}}^{2} + \int_{\mathbb{R}} \phi^{p+1} \, dx = \int_\mathbb{R} \phi^{q+1} \, dx. \label{eq:GS0210;gkdv_st}
    \end{equation}
    Considering the graph of the function \eqref{eq:GS0100;gkdv_st}, we can conclude \eqref{eq:GS0200_gkdv_st}.
    Therefore, by the Lagrange multiplier theorem, there exists $\mu \in \mathbb{R}$ such that $S_{c}'(\phi) = \mu K_{c}'(\phi)$.
    Then we have
    \begin{equation}
        \mu \langle K_{c}'(\phi), \phi \rangle = \langle S_{c}'(\phi), \phi \rangle = K_{c}(\phi) = 0,
    \end{equation}
    which implies $\mu = 0$ and then $S_{c}'(\phi) = 0$.
    Furthermore, from the definition of $\mathcal{M}_{c}$, it holds that $S_{c}(\phi) = d(c) \leq S_{c}(v)$ for any nontrivial solution $v \in H^{\sigma/2}(\mathbb{R})$ to \eqref{eq:Intro0300;gkdv_st}, which means that $\phi \in \mathcal{G}_{c}$.

    Next, we show that $\mathcal{G}_{c} \subset \mathcal{M}_{c}$.
    Since $\mathcal{M}_{c} \neq \emptyset$, we can take some $v \in \mathcal{M}_{c}$.
    Here we let $\phi \in \mathcal{G}_{c}$.
    Then, since $v \in \mathcal{G}_{c}$ as shown above, we see that $S_{c}(\phi) \leq S_{c}(v) = d(c)$.
    By $K_{c}(\phi) = 0$, we have $d(c) \leq S_{c}(\phi)$.
    Therefore, we obtain $S_{c}(\phi) = d(c)$, which means that $\phi \in \mathcal{M}_{c}$.

    Hence, the proof is completed.
\end{proof}

Thanks to Lemma \ref{lem:existence_gs_GM;gkdv_st}, in order to prove the existence of ground state solutions to \eqref{eq:Intro0300;gkdv_st}, we shall show that $\mathcal{M}_{c} \neq \emptyset$.
We prove this in what follows.

Here we introduce some auxiliary functionals.
For $v \in H^{\sigma/2}(\mathbb{R})$, we put
\begin{align}
    I_{c}(v) &= S_{c}(v) - \frac{1}{q+1} K_{c}(v) \\
    &= \left( \frac{1}{2} - \frac{1}{q+1} \right) \| v \|_{H^{\sigma/2}_{c}}^{2} - a \left( \frac{1}{p+1} - \frac{1}{q+1} \right) \int_{\mathbb{R}} v^{p+1} \, dx, \label{eq:GS0220;gkdv_st} \\
    J_{c}(v) &= S_{c}(v) - \frac{1}{p+1} K_{c}(v) \\    
    &= \left( \frac{1}{2} - \frac{1}{p+1} \right) \| v \|_{H^{\sigma/2}_{c}}^{2} + \left( \frac{1}{p+1} - \frac{1}{q+1} \right) \int_{\mathbb{R}} v^{q+1} \, dx. \label{eq:GS0230;gkdv_st}
\end{align}
We can easily see that
\begin{align}
    d(c) &= \inf \{ S_{c}(v): v \in H^{\sigma/2}(\mathbb{R}) \setminus \{0\}, \ K_{c}(v) = 0 \} \\
    &= \inf \{ I_{c}(v): v \in H^{\sigma/2}(\mathbb{R}) \setminus \{0\}, \ K_{c}(v) = 0 \} \label{eq:GS0300;gkdv_st} \\
    &= \inf \{ J_{c}(v): v \in H^{\sigma/2}(\mathbb{R}) \setminus \{0\}, \ K_{c}(v) = 0 \}. \label{eq:GS0400;gkdv_st}
\end{align}

\begin{lem} \label{lem:GS_ineq_IJ;gkdv_st}
    \begin{enumerate}[label={\rm (\roman*)} \ ]
        \item Assume condition {\rm (I)} or {\rm (II-2)} in Theorem \ref{thm:existence_of_GS;gkdv_st}.
        Then $J_{c}(v) > d(c)$ holds for all $v \in H^{\sigma/2}(\mathbb{R})$ satisfying $K_{c}(v) < 0$.
        \item Assume condition {\rm (II-1)} in Theorem \ref{thm:existence_of_GS;gkdv_st}.
        Then $I_{c}(v) > d(c)$ holds for all $v \in H^{\sigma/2}(\mathbb{R})$ satisfying $K_{c}(v) < 0$.
    \end{enumerate}
\end{lem}
\begin{proof}
    (i) \ Let $v \in H^{\sigma/2}(\mathbb{R})$ satisfying $K_{c}(v) < 0$.
    We note that $v \neq 0$ and that $\int_{\mathbb{R}} v^{q+1} \, dx > 0$ as $q$ is odd.
    Then, considering the graph of the function $(0,\infty) \ni \lambda \mapsto J_{c}(\lambda v)$, we find some $\lambda_{0} \in (0, 1)$ such that $K_{c}(\lambda_{0} v) = 0$.
    By \eqref{eq:GS0400;gkdv_st}, we obtain $d(c) \leq J_{c}(\lambda_{0} v) < J_{c} (v)$.

    (ii) \ In this case, the functional $I_{c}$ is written as
    \begin{equation}
        I_{c}(v) = \left( \frac{1}{2} - \frac{1}{q+1} \right) \| v \|_{H^{\sigma/2}_{c}}^{2} + \left( \frac{1}{p+1} - \frac{1}{q+1} \right) \int_{\mathbb{R}} v^{p+1} \, dx.
    \end{equation}
    Furthermore, similarly to \eqref{eq:GS0210;gkdv_st}, we can see that $\int_{\mathbb{R}} v^{p+1} \, dx > 0$ holds for all $v \in H^{\sigma/2}(\mathbb{R})$ satisfying $K_{c}(v) < 0$.
    Then we can prove the statement with the almost the same way as (i).
\end{proof}

\begin{lem} \label{lem:GS_positivity_d;gkdv_st}
    It holds that $d(c) > 0$.
\end{lem}
\begin{proof}
    We shall show that there exists some $C > 0$ such that $I_{c}(v) \geq C$ or $J_{c}(v) \geq C$ hold for all $v \in H^{\sigma/2}(\mathbb{R}) \setminus \{0\}$ satisfying $K_{c}(v) = 0$.

    Here we let $v \in H^{\sigma/2}(\mathbb{R}) \setminus \{0\}$ satisfy $K_{c}(v) = 0$.
    By the Sobolev embedding, we see that
    \begin{align}
        0 = K_{c}(v) &= \| v \|_{H^{\sigma/2}_{c}}^{2} - a \int_{\mathbb{R}} v^{p+1} \, dx - \int_{\mathbb{R}} v^{q+1} \, dx \\
        &\geq \| v \|_{H^{\sigma/2}_{c}}^{2} \left\{ 1 - C \left( \| v \|_{H^{\sigma/2}_{c}}^{p-1} + \| v \|_{H^{\sigma/2}_{c}}^{q-1} \right) \right\},
    \end{align}
    with some constant $C > 0$.
    Then we see that $1 \leq C (\| v \|_{H^{\sigma/2}_{c}}^{p-1} + \| v \|_{H^{\sigma/2}_{c}}^{q-1})$.
    Noting that $p < q$, we have
    \begin{equation}
        1 \leq \left\{
            \begin{aligned}
                C_{1} \| v \|_{H^{\sigma/2}_{c}}^{p-1}, & \quad \text{if} \ \| v \|_{H^{\sigma/2}_{c}} \geq 1, \\
                C_{2} \| v \|_{H^{\sigma/2}_{c}}^{q-1}, & \quad \text{if} \ \| v \|_{H^{\sigma/2}_{c}} \leq 1.
            \end{aligned}
        \right.
    \end{equation}
    Here we set $C_{0} \coloneq \min \{ C_{1}^{-2/(p-1)}, \ C_{2}^{-2/(q-1)} \}$ so that we obtain $\| v \|_{H^{\sigma/2}_{c}}^{2} \geq C_{0}$ for all $v \in H^{\sigma/2}(\mathbb{R})$ satisfying $K_{c}(v) = 0$.

    In cases (I) and (II-2) in Theorem \ref{thm:existence_of_GS;gkdv_st}, noting that $\int_{\mathbb{R}} v^{q+1} \, dx \geq 0$, we have
    \begin{align}
        J_{c}(v) &= \left( \frac{1}{2} - \frac{1}{p+1} \right) \| v \|_{H^{\sigma/2}_{c}}^{2} + \left( \frac{1}{p+1} - \frac{1}{q+1} \right) \int_{\mathbb{R}} v^{q+1} \, dx \\
        &\geq \left( \frac{1}{2} - \frac{1}{p+1} \right) \| v \|_{H^{\sigma/2}_{c}}^{2} \geq \left( \frac{1}{2} - \frac{1}{p+1} \right) C_{0},
    \end{align}
    for all $v \in H^{\sigma/2}(\mathbb{R})$ satisfying $K_{c}(v) = 0$.

    We can see that it holds under the condition (II-1) in Theorem \ref{thm:existence_of_GS;gkdv_st} with a similar discussion that
    \begin{equation}
        I_{c}(v) \geq \left( \frac{1}{2} - \frac{1}{q+1} \right) C_{0}
    \end{equation}
    holds for all $v \in H^{\sigma/2}(\mathbb{R})$ satisfying $K_{c}(v) = 0$.
\end{proof}

To prove that $\mathcal{M}_{c} \neq \emptyset$, we need two more lemmas.

\begin{lem} \label{lem:GS_Lieb;gkdv_st}
    Let $(v_{n})_{n}$ be a bounded sequence in $H^{\sigma/2}(\mathbb{R})$.
    Assume that there exists some $r \in (2, \infty)$ such that $\inf_{n \in \mathbb{N}} \| v_{n} \|_{L^{r}} > 0$.
    Then there exists $(z_{n})_{n} \subset \mathbb{R}$ such that, up to a subsequence, there exists some $v_{0} \in H^{\sigma/2}(\mathbb{R}) \setminus \{0\}$ such that $u_{n}(\cdot + z_{n}) \rightharpoonup v_{0}$ weakly in $H^{\sigma/2}(\mathbb{R})$.
\end{lem}

The similar statement for a sequence in $H^{1}(\mathbb{R})$ was proved by Lieb~\cite{Lieb}.
For the proof, see \cite[Appendix A]{Kokubu2024}.

\begin{lem}[Brezis--Lieb~\cite{Brezis-Lieb}] \label{eq:GS_Brezis-Lieb;gkdv_st}
    Let $r \in (1,\infty)$ and $(v_{n})_{n}$ be a bounded sequence in $L^{r}(\mathbb{R})$.
    Assume that $v_{n}(x) \rightarrow v(x)$ a.e.\ in $\mathbb{R}$ with some measurable function $v$.
    Then it holds that $v \in L^{r}(\mathbb{R})$ and that
    \begin{equation}
        \lim_{n \rightarrow \infty} \left( \int_{\mathbb{R}} |v_{n}|^{r} \, dx - \int_{\mathbb{R}} |v_{n} - v|^{r} \, dx \right) = \int_{\mathbb{R}} |v|^{r} \, dx.
    \end{equation}
\end{lem}

Now we prove that $\mathcal{M}_{c} \neq \emptyset$ holds.
\begin{proof}
    Let $(v_{n})_{n} \subset H^{\sigma/2}(\mathbb{R})$ satisfy $S_{c}(v_{n}) \rightarrow d(c)$ and $K_{c}(v) \rightarrow 0$.
    Then we have
    \begin{align}
        & I_{c}(v_{n}) \rightarrow d(c), \label{eq:GS0500;gkdv_st} \\
        & J_{c}(v_{n}) \rightarrow d(c), \label{eq:GS0510;gkdv_st} \\
        & a \left( \frac{1}{2} - \frac{1}{p+1} \right) \int_{\mathbb{R}} v_{n}^{p+1} \, dx + \left( \frac{1}{2} - \frac{1}{q+1} \right) \int_{\mathbb{R}} v_{n}^{q+1} \, dx = S_{c}(v_{n}) - \frac{1}{2} K_{c}(v_{n}) \rightarrow d(c). \label{eq:GS0520;gkdv_st}
    \end{align}
    We can see it from \eqref{eq:GS0500;gkdv_st} or \eqref{eq:GS0510;gkdv_st} that $(v_{n})_{n}$ is bounded in $H^{\sigma/2}(\mathbb{R})$.
    Moreover, by \eqref{eq:GS0520;gkdv_st}, we have $\inf_{n \in \mathbb{R}} \| v_{n} \|_{L^{p+1}} > 0$.
    Then, by Lemma \ref{lem:GS_Lieb;gkdv_st}, there exists $(z_{n})_{n} \subset \mathbb{R}$ and $v_{0} \in H^{\sigma/2}(\mathbb{R}) \setminus \{0\}$ such that $v_{n}(\cdot + z_{n}) \rightharpoonup v_{0}$ weakly in $H^{\sigma/2}(\mathbb{R})$, up to a subsequence.
    Here we put $w_{n} \coloneq v_{n}(\cdot + z_{n})$.
    Then we may assume that $w_{n}(x) \rightarrow v_{0}(x)$ a.e.\ in $\mathbb{R}$ due to the weak convergence and the Rellich compact embedding.
    Therefore, applying Lemma \ref{eq:GS_Brezis-Lieb;gkdv_st} to the sequence $(w_{n})_{n}$, we have
    \begin{align}
        & I_{c}(w_{n}) - I_{c}(w_{n} - v_{0}) \rightarrow I_{c}(v_{0}), \label{eq:GS0530;gkdv_st} \\
        & J_{c}(w_{n}) - J_{c}(w_{n} - v_{0}) \rightarrow J_{c}(v_{0}), \label{eq:GS0540;gkdv_st} \\
        & K_{c}(w_{n}) - K_{c}(w_{n} - v_{0}) \rightarrow K_{c}(v_{0}). \label{eq:GS0550;gkdv_st}
    \end{align}

    In cases (I) and (II-2) in Theorem \ref{thm:existence_of_GS;gkdv_st}, we obtain it from \eqref{eq:GS0510;gkdv_st} and \eqref{eq:GS0540;gkdv_st} that
    \begin{align}
        \lim_{n \rightarrow \infty} J_{c}(w_{n} - v_{0}) &= \lim_{n \rightarrow \infty} J_{c}(w_{n}) - J_{c}(v_{0}) \\
        &< \lim_{n \rightarrow \infty} J_{c}(w_{n}) = \lim_{n \rightarrow \infty} J_{c}(v_{n}) = d(c),
    \end{align}
    which implies that $J_{c}(w_{n} - v_{0}) \leq d(c)$ holds for $n \in \mathbb{N}$ sufficiently large.
    Then, by Lemma \ref{lem:GS_positivity_d;gkdv_st}, we have $K_{c}(w_{n} - v_{0}) \geq 0$ for large $n \in \mathbb{N}$.
    Moreover, by \eqref{eq:GS0550;gkdv_st}, we see that
    \begin{equation}
        K_{c}(v_{0}) = \lim_{n \rightarrow \infty} K_{c}(w_{n}) - \lim_{n \rightarrow \infty} K_{c}(w_{n} - v_{0}) \leq 0. \label{eq:GS0560;gkdv_st}
    \end{equation}
    Therefore, using Lemma \ref{lem:GS_positivity_d;gkdv_st} again, we have
    \begin{equation}
        d(c) \leq J_{c}(v_{0}) \leq \liminf_{n \rightarrow \infty} J_{c}(w_{n}) = \lim_{n \rightarrow \infty} J_{c}(v_{n}) = d(c),
    \end{equation}
    that is, $J_{c}(v_{0}) = d(c)$.
    Combining this with Lemma \ref{lem:GS_positivity_d;gkdv_st}, we obtain $K_{c}(v) \geq 0$.
    This and \eqref{eq:GS0560;gkdv_st} yield that $K_{c}(v_{0}) = 0$.
    Thus, we conclude that $v_{0} \in \mathcal{M}_{c}$.

    In case (II-1), we can similarly conclude that $\mathcal{M}_{c} \neq \emptyset$ using \eqref{eq:GS0530;gkdv_st} instead of \eqref{eq:GS0540;gkdv_st}.

    Hence, the proof is accomplished.
\end{proof}

Here we observe solutions to the stationary problems \eqref{eq:Intro0300;gkdv_st} further.
\begin{lem} \label{lem:GS_regularlity_of_sol;gkdv_st}
    Let $\phi \in H^{\sigma/2}(\mathbb{R})$ be a nontrivial solution to \eqref{eq:Intro0300;gkdv_st}.
    Then $\phi \in H^{\infty}(\mathbb{R})$ holds.
\end{lem}
\begin{proof}
    Let $\phi \in H^{\sigma/2}(\mathbb{R})$ be a nontrivial solution to \eqref{eq:Intro0300;gkdv_st}.
    Then, taking the Fourier transform, we obtain
    \begin{equation}
        \hat{\phi}(\xi) = \frac{1}{|\xi|^{\sigma} + c} \mathscr{F}\left[ a\phi^{p} + \phi^{q} \right].
    \end{equation}
    First, by the Sobolev embedding, we have
    \begin{align}
        \| D_{x}^{\sigma} \phi \|_{L^{2}} = \| |\xi|^{\sigma} \hat{\phi} \|_{L^{2}} &= \left\| \frac{|\xi|^{\sigma}}{|\xi|^{\sigma} + c} \mathscr{F}\left[ a \phi^{p} + \phi^{q} \right] \right\|_{L^{2}} \\
        &\leq \| \mathscr{F}\left[ \phi^{p} \right] \|_{L^{2}} + \| \mathscr{F}\left[ \phi^{q} \right] \|_{L^{2}} \\
        &=  \| \phi \|_{L^{2p}}^{p} + \| \phi \|_{L^{2q}}^{q} \\
        &\leq C \left( \| \phi \|_{H^{\sigma/2}}^{p} + \| \phi \|_{H^{\sigma/2}}^{q} \right),
    \end{align}
    which implies that $\phi \in H^{\sigma}(\mathbb{R}) \hookrightarrow H^{1}(\mathbb{R})$.
    Moreover, since $p, q \in \mathbb{N}$, we see that $\| \phi^{r} \|_{H^{1}} \leq C \| \phi \|_{H^{1}}^{r}$ for $r \in \{p, q$\}.
    Then, we have
    \begin{align}
        \| D_{x}^{\sigma + 1} \phi \|_{L^{2}} = \| |\xi|^{\sigma + 1} \hat{\phi} \|_{L^{2}} &= \left\| \frac{|\xi|^{\sigma + 1}}{|\xi|^{\sigma} + c} \mathscr{F}\left[ a \phi^{p} + \phi^{q} \right] \right\|_{L^{2}} \\
        &\leq \| |\xi| \mathscr{F} \left[ \phi^{p} \right] \|_{L^{2}} + \| |\xi| \mathscr{F} \left[ \phi^{q} \right] \|_{L^{2}} \\
        &\leq \| \phi^{p} \|_{H^{1}} + \| \phi^{q} \|_{H^{1}} \\
        &\leq C \left( \| \phi \|_{H^{1}}^{p} + \| \phi \|_{H^{1}}^{q} \right),
    \end{align}
    which implies that $\phi \in H^{\sigma+1}(\mathbb{R})$.
    Inductively, we can obtain $\phi \in H^{\sigma + j}(\mathbb{R})$ for all $j \in \mathbb{N}$, which completes the proof.
\end{proof}

\subsection{Evenness, positivity, and negativity}
In this subsection, we observe properties of ground state solutions to \eqref{eq:Intro0300;gkdv_st}.

First, we recall the symmetric decreasing rearrangement of a nonnegative function.
Let $v \in H^{\sigma/2}(\mathbb{R})$.
Then the following inequality holds:
\begin{equation}
    \| D_{x}^{\sigma/2} |v| \|_{L^{2}} \leq \| D_{x}^{\sigma/2} v \|_{L^{2}}, \label{eq:GS0600;gkdv_st}
\end{equation}
which implies that $|v| \in H^{\sigma/2}(\mathbb{R})$.
Moreover, we can define the symmetric decreasing rearrangement of $|v|$ with $|v|^{\ast}$ denoting it, and obtain
\begin{equation}
    \| D_{x}^{\sigma/2} |v|^{\ast} \|_{L^{2}} \leq \| D_{x}^{\sigma/2} |v| \|_{L^{2}}, \label{eq:GS0610;gkdv_st}
\end{equation}
We can prove these inequalities similarly to \cite[Lemma 8.15]{Angulo}.
In addition, we note that $\| v^{\ast} \|_{L^{\gamma}} = \| v \|_{L^{\gamma}}$ holds for any nonnegative function $v \in L^{\gamma}(\mathbb{R})$ and any $\gamma \in [1, \infty)$.

\begin{lem} \label{lem:GS_nonnegative_even;gkdv_st}
    Assume condition {\rm (I)} or {\rm (II-1)} in Theorem \ref{thm:existence_of_GS;gkdv_st}.
    Let $\phi \in \mathcal{G}_{c}$.
    Then $|\phi|^{\ast} \in \mathcal{G}_{c}$ holds.
\end{lem}
\begin{proof}
    Let $\phi \in \mathcal{G}_{c}$.
    Then $\phi$ satisfies that $S_{c}(\phi) = d(c)$ and $K_{c}(\phi) = 0$.
    By \eqref{eq:GS0600;gkdv_st} and \eqref{eq:GS0610;gkdv_st}, we have
    \begin{align}
        S_{c}(|\phi|^{\ast}) &= \frac{1}{2} \| |\phi|^{\ast} \|_{H^{\sigma/2}_{c}}^{2} - \frac{a}{p+1} \int_{\mathbb{R}} (|\phi|^{\ast})^{p+1} \, dx - \frac{1}{q+1} \int_{\mathbb{R}} (|\phi|^{\ast})^{q+1} \, dx \\
        &\leq \frac{1}{2} \| \phi \|_{H^{\sigma/2}_{c}}^{2} - \frac{a}{p+1} \int_{\mathbb{R}} |\phi|^{p+1} \, dx - \frac{1}{q+1} \int_{\mathbb{R}} |\phi|^{q+1} \, dx \\
        &\leq \frac{1}{2} \| \phi \|_{H^{\sigma/2}_{c}}^{2} - \frac{a}{p+1} \int_{\mathbb{R}} \phi^{p+1} \, dx - \frac{1}{q+1} \int_{\mathbb{R}} \phi^{q+1} \, dx =S_{c}(\phi) = d(c), \label{eq:GS0620;gkdv_st}
    \end{align}
    and similarly,
    \begin{align}
        & K_{c}(|\phi|^{\ast}) \leq K_{c}(\phi) = 0, \label{eq:GS0630;gkdv_st} \\
        & J_{c}(|\phi|^{\ast}) \leq J_{c}(\phi) = d(c). \label{eq:GS0640;gkdv_st}
    \end{align}
    By \eqref{eq:GS0640;gkdv_st} and Lemma \ref{lem:GS_ineq_IJ;gkdv_st}, we obtain $K_{c}(|\phi|^{\ast}) \geq 0$.
    This and \eqref{eq:GS0630;gkdv_st} give that $K_{c}(|\phi|^{\ast}) = 0$.
    Therefore, by \eqref{eq:GS0620;gkdv_st} and the definition of $d(c)$, we see that $S_{c}(|\phi|^{\ast}) = d(c)$, which implies that $|\phi|^{\ast} \in \mathcal{M}_{c} = \mathcal{G}_{c}$.
\end{proof}

As a consequence of Lemma \ref{lem:GS_nonnegative_even;gkdv_st}, the stationary problem \eqref{eq:Intro0300;gkdv_st} has an even and nonnegative ground state solution in cases (I) or (II-1) in Theorem \ref{thm:existence_of_GS;gkdv_st}.

In case (II-2) in Theorem \ref{thm:existence_of_GS;gkdv_st}, we can see the existence of an even and nonpositive ground state solution to \eqref{eq:Intro0300;gkdv_st}.
\begin{lem} \label{lem:GS_even_nonpositive;gkdv_st}
    Assume condition {\rm (II-2)} in Theorem \ref{thm:existence_of_GS;gkdv_st}.
    Let $\phi \in \mathcal{G}_{c}$.
    Then $-|\phi|^{\ast} \in \mathcal{G}_{c}$ holds.
\end{lem}
\begin{proof}
    Since $p$ is even, and $q$ is odd, we see that
    \begin{align}
        & \int_{\mathbb{R}} (-|\phi|^{\ast})^{p+1} \, dx = - \int_{\mathbb{R}} (|\phi|^{\ast})^{p+1} \, dx = - \int_{\mathbb{R}} |\phi|^{p+1} \, dx, \\
        & \int_{\mathbb{R}} (-|\phi|^{\ast})^{q+1} \, dx = \int_{\mathbb{R}} (|\phi|^{\ast})^{q+1} \, dx = \int_{\mathbb{R}} |\phi|^{q+1} \, dx = \int_{\mathbb{R}} \phi^{q+1} \, dx.
    \end{align}
    Then, similarly to Lemma \ref{lem:GS_nonnegative_even;gkdv_st}, we obtain
    \begin{align}
        S_{c}(-|\phi|^{\ast}) &= \frac{1}{2} \| |\phi|^{\ast} \|_{H^{\sigma/2}_{c}}^{2} - \frac{1}{p+1} \int_{\mathbb{R}} (|\phi|^{\ast})^{p+1} \, dx - \frac{1}{q+1} \int_{\mathbb{R}} (|\phi|^{\ast})^{q+1} \, dx \\
        &\leq \frac{1}{2} \| \phi \|_{H^{\sigma/2}_{c}}^{2} - \frac{1}{p+1} \int_{\mathbb{R}} |\phi|^{p+1} \, dx - \frac{1}{q+1} \int_{\mathbb{R}} |\phi|^{q+1} \, dx \\
        &\leq \frac{1}{2} \| \phi \|_{H^{\sigma/2}_{c}}^{2} + \frac{1}{p+1} \int_{\mathbb{R}} \phi^{p+1} \, dx - \frac{1}{q+1} \int_{\mathbb{R}} \phi^{q+1} \, dx = S_{c}(\phi) = d(c), \\
        K_{c}(-|\phi|^{\ast}) &\leq K_{c}(\phi) = 0, \\
        I_{c}(-|\phi|^{\ast}) &\leq I_{c}(\phi) = d(c).
    \end{align}
    Therefore, we can obtain $K_{c}(-|\phi|^{\ast}) = 0$ and $S_{c}(-|\phi|^{\ast}) = d(c)$ with similar way to Lemma \ref{lem:GS_nonnegative_even;gkdv_st}, and conclude that $-|\phi|^{\ast} \in \mathcal{M}_{c} = \mathcal{G}_{c}$.
\end{proof}

At the end of this subsection, we consider the positivity or negativity of ground state solutions.
First, for $\nu > 0$, we define a function $N_{\nu}^{\sigma}\colon \mathbb{R} \rightarrow \mathbb{R}$ as
\begin{equation}
    N_{\nu}^{\sigma}(x) \coloneq \frac{1}{\sqrt{2 \pi}} \int_{\mathbb{R}} \frac{1}{|\xi|^{\sigma} + \nu} e^{i \xi x} \, dx = \mathscr{F}^{-1} \left[ \frac{1}{|\xi|^{\sigma} + \nu} \right](x). \label{eq:GS0650;gkdv_st}
\end{equation}
It is known that $N_{\nu}^{\sigma}$ is positive, even, and strictly decreasing in $|x|$.
For details, see e.g.\ \cite[Appendix A]{Frank-Lenzmann}.

\begin{lem} \label{lem:GS_positiveness;gkdv_st}
    Assume conditions {\rm (I)} or {\rm (II-1)} in Theorem \ref{thm:existence_of_GS;gkdv_st}.
    Let $\phi \in H^{\sigma/2}(\mathbb{R})$ be a nonnegative solution to \eqref{eq:Intro0300;gkdv_st}.
    Then $\phi$ is strictly positive.
\end{lem}
\begin{proof}
    Let $\phi \in H^{\sigma/2}(\mathbb{R})$ be a nonnegative solution to \eqref{eq:Intro0300;gkdv_st}.
    By Lemma \ref{lem:GS_regularlity_of_sol;gkdv_st}, we see that $\phi \in C^{\infty}(\mathbb{R}) \cap L^{\infty}(\mathbb{R})$.
    
    Here we set
    \begin{equation}
        \lambda \coloneq \sup_{x \in \mathbb{R}}\left( a \phi^{p-1}(x) + \phi^{q-1}(x) \right) + 1.
    \end{equation}
    Then, we see that $\lambda > 0$ and
    \begin{equation}
        \lambda - (a \phi^{p-1}(y) + \phi^{q-1}(y)) \geq 1 \label{eq:GS0660;gkdv_st}
    \end{equation}
    holds for all $y \in \mathbb{R}$.
    Now adding $\lambda \phi$ to both sides of \eqref{eq:Intro0300;gkdv_st} and using the function $N_{\nu}^{\sigma}$ with $\nu = c + \lambda$, we obtain
    \begin{equation}
        \phi(x) = \left( N_{c + \lambda}^{\sigma} \ast \left( \lambda \phi - (a \phi^{p} + \phi^{q}) \right) \right) (x) = \int_{\mathbb{R}} N_{c + \lambda}^{\sigma}(x-y)(\lambda - (a \phi^{p-1}(y) + \phi^{q-1}(y)))\phi(y) \, dy. \label{eq:GS0670;gkdv_st}
    \end{equation}
    Suppose that there exists some $x_{0} \in \mathbb{R}$ such that $\phi(x_{0}) = 0$.
    By \eqref{eq:GS0670;gkdv_st}, we obtain
    \begin{equation}
        \int_{\mathbb{R}} N_{c + \lambda}^{\sigma}(x_{0} - y)(\lambda - (a \phi^{p-1}(y) + \phi^{q-1}(y)))\phi(y) \, dy = 0.
    \end{equation}
    Since $N_{c + \lambda}^{\sigma} > 0$ in $\mathbb{R}$, we see it from \eqref{eq:GS0660;gkdv_st} that $\phi \equiv 0$ in $\mathbb{R}$, which contradicts to that $\phi$ is nontrivial.

    Hence, the proof is completed.
\end{proof}

With almost the same method as Lemma \ref{lem:GS_positiveness;gkdv_st}, we can show the following statement.
\begin{lem} \label{lem:GS_negativeness;gkdv_st}
    Assume condition {\rm (II-2)} in Theorem \ref{thm:existence_of_GS;gkdv_st}.
    Let $\phi \in H^{\sigma/2}(\mathbb{R})$ be a nonpositive solution to \eqref{eq:Intro0300;gkdv_st}.
    Then $\phi$ is strictly negative.
\end{lem}

\section{Sufficient conditions for the stability of travelling wave solutions} \label{section:stability;gkdv_st}
In this section, we consider sufficient conditions for Theorem \ref{thm:stability;gkdv_st} following the method by Grillakis--Shatah--Strauss~\cite{GSS87}.

For $c > 0$, we let $\phi_{c} \in H^{\sigma/2}(\mathbb{R})$ be a ground state solution to \eqref{eq:Intro0300;gkdv_st} obtained in Theorem \ref{thm:existence_of_GS;gkdv_st}.
According to Theorem 3.5 of \cite{GSS87}, the following statement (A) is sufficient for travelling wave solutions to be stable:
\begin{itemize}
    \item[(A)] 
        \textit{
        There exist $C_{0} > 0$ and $\varepsilon_{0} > 0$ shch that
            \begin{equation}
                E(v) - E(\phi_{c}) \geq C_{0} \inf_{y \in \mathbb{R}} \| v - \phi_{c}(\cdot - y) \|_{H^{\sigma/2}}^{2} \label{eq:stability0100;gkdv_st}
            \end{equation}
        holds for all $u \in U_{\varepsilon_{0}}(\phi_{c})$ satisfying $M(u) = M(\phi_{c})$.
        }
\end{itemize}
We can show that statement (A) holds if we assume the following statement:
\begin{itemize}
    \item[(B)]
        \textit{
            There exists $C_{1} > 0$ which satisfies the following statement:
        \begin{equation}
            \langle S_{c}''(\phi_{c})v, v \rangle \geq C_{1} \| v \|_{H^{\sigma/2}}^{2} \label{eq:stability0200;gkdv_st}
        \end{equation}
        holds for all $v \in H^{\sigma/2}(\mathbb{R})$ satisfying $(v, \phi_{c})_{L^{2}} = (v, \partial_{x} \phi_{c})_{L^{2}} = 0$.
        }
\end{itemize}
\begin{proof}[Proof that (B) $\Longrightarrow$ (A)]
    First, with the idea of Cipolatti~\cite[Lemma 3.4]{Cipolatti_1993}, we can see that there exists $\varepsilon_{1} > 0$ such that, for all $\varepsilon \in (0, \varepsilon_{1})$ and $u \in U_{\varepsilon}(\phi_{c})$, there exists $\tilde{y} \in \mathbb{R}$ such that
    \begin{equation}
        \| u - \phi_{c}(\cdot - \tilde{y}) \|_{H^{\sigma/2}} = \min_{y \in \mathbb{R}} \| u - \phi_{c}(\cdot - y) \|_{H^{\sigma/2}}. \label{eq:B_000;BOst}
    \end{equation}

    Here we put $v \coloneqq u (\cdot + \tilde{y}) - \phi_{c}$.
    Then, by the Taylor expansion, we have
    \begin{align}
        S_{c}(u) &= S_{c}(u(\cdot + \tilde{y})) = S_{c}(v + \phi_{c}) \\
        &= S_{c}(\phi_{c}) + \langle S_{c}'(\phi_{c}), v \rangle + \frac{1}{2} \langle S_{c}''(\phi_{c})v,v \rangle + o(\| v \|_{H^{\sigma/2}}^{2}), \\
        M(\phi_{c}) &= M(u) = M(u(\cdot + y)) = M(\phi_{c} + v) \\
        &= M(\phi_{c}) + \langle M'(\phi_{c}), v \rangle + O(\| v \|_{H^{\sigma/2}}^{2}).
    \end{align}
    Therefore, we have it from $S_{c}'(\phi_{c}) = 0$ that
    \begin{equation}
        S_{c}(u) - S_{c}(\phi_{c}) = \frac{1}{2}\langle S_{c}''(\phi_{c})v,v \rangle + o(\| v \|_{H^{\sigma/2}}^{2}) \label{eq:B_001;BOst}
    \end{equation}
    and
    \begin{equation}
        \langle M'(\phi_{c}), v \rangle = O(\| v \|_{H^{\sigma/2}}^{2}). \label{eq:B_002;BOst}
    \end{equation}

    Next, noting that since $\| \phi_{c}(\cdot + y) \|_{L^{2}} = \| \phi_{c} \|_{L^{2}}$ holds for all $y \in \mathbb{R}$, we can see that $(\phi_{c}, \partial_{x}\phi_{c})_{L^{2}} = 0$.
    From this, we can decompose $v$ as $v = k \phi_{c} + l \partial_{x}\phi_{c} + w$ with some $k, l \in \mathbb{R}$ and some $w \in H^{\sigma/2}(\mathbb{R})$ which satisfies $(w, \phi_{c})_{L^{2}} = (w, \partial_{x}\phi_{c})_{L^{2}} = 0$.
    Then, we have
    \begin{align}
        \langle M'(\phi_{c}), v \rangle &= (\phi_{c}, k \phi_{c} + l \partial_{x}\phi_{c} + w)_{L^{2}} = k \| \phi_{c} \|_{L^{2}}^{2}. \label{eq:B_003;BOst}
    \end{align}
    Combining \eqref{eq:B_002;BOst} and \eqref{eq:B_003;BOst} yields $k = O(\| v \|_{H^{\sigma/2}}^{2})$.

    In addition, we can see that $(v, \partial_{x}\phi_{c})_{H^{\sigma/2}} = 0$.
    Indeed, we put
    \begin{equation}
        g(y) \coloneqq \| u - \phi_{c}(\cdot - y) \|_{H^{\sigma/2}}^{2} = \| u \|_{H^{\sigma/2}}^{2} - 2(u, \phi_{c}(\cdot - y))_{H^{\sigma/2}} + \| \phi_{c} \|_{H^{\sigma/2}}^{2}, \quad y\in \mathbb{R}.
    \end{equation}
    Since $\| \phi_{c}(\cdot + y) \|_{H^{\sigma/2}} = \| \phi_{c} \|_{H^{\sigma/2}}$ holds for all $y\in \mathbb{R}$, we have $(\phi_{c}, \partial_{x}\phi_{c})_{H^{\sigma/2}} = 0$. 
    This and \eqref{eq:B_000;BOst} give
    \begin{align}
        0 = \partial_{y} g(\tilde{y}) = 2(u, \partial_{x}\phi_{c}(\cdot - \tilde{y}))_{H^{\sigma/2}} &= 2 (u(\cdot + \tilde{y}), \partial_{x}\phi_{c})_{H^{\sigma/2}} \\
        &= 2(v + \phi_{c}, \partial_{x}\phi_{c})_{H^{\sigma/2}} = 2(v, \partial_{x}\phi_{c})_{H^{\sigma/2}}.
    \end{align}
    Moreover, we see that
    \begin{align}
        (v, \partial_{x}\phi_{c})_{H^{\sigma/2}} = l \| \partial_{x}\phi_{c} \|_{H^{\sigma/2}}^{2} + (w, \partial_{x}\phi_{c})_{H^{\sigma/2}}.
    \end{align}
    Therefore, we obtain
    \begin{align}
        |l| \| \partial_{x} \phi_{c} \|_{H^{\sigma/2}}^{2} \leq \| w \|_{H^{\sigma/2}} \| \partial_{x}\phi_{c} \|_{H^{\sigma/2}},
    \end{align}
    that is, $|l| \| \partial_{x}\phi_{c} \|_{H^{\sigma/2}} \leq \| w \|_{H^{\sigma/2}}$.
    Then we have
    \begin{align}
        \| v \|_{H^{\sigma/2}} &\leq |k| \| \phi_{c} \|_{H^{\sigma/2}} + |l| \| \partial_{x}\phi_{c} \|_{H^{\sigma/2}} + \| w \|_{H^{\sigma/2}} \\
        &\leq |k| \| \phi_{c} \|_{H^{\sigma/2}} + 2 \| w \|_{H^{\sigma/2}} = 2 \| w \|_{H^{\sigma/2}} + O(\| v \|_{H^{\sigma/2}}^{2}),
    \end{align}
    which implies
    \begin{equation}
        \frac{1}{4}\| v \|_{H^{\sigma/2}}^{2} + O(\| v \|_{H^{\sigma/2}}^{3}) \leq \| w \|_{H^{\sigma/2}}^{2}. \label{eq:B_004;BOst}
    \end{equation}

    Next, since $S_{c}'(\phi_{c}(\cdot + y)) = 0$ for all $y \in \mathbb{R}$, we can see that $\left. \partial_{y}\{ S_{c}'(\phi_{c}(\cdot + y)) \} \right|_{c = 0} = S_{c}''(\phi_{c}) \partial_{x}\phi_{c} = 0$.
    This yields that
    \begin{align}
        \langle S_{c}''(\phi_{c}) w, w \rangle &= \langle S_{c}''(\phi_{c})v , v \rangle - 2k \langle S_{c}''(\phi_{c})v, \phi_{c} \rangle + k^{2} \langle S_{c}''(\phi_{c}) \phi_{c}, \phi_{c} \rangle \\
        &= \langle S_{c}''(\phi_{c})v,v \rangle + O(\| v \|_{H^{\sigma/2}}^{3}). \label{eq:B_005;BOst}
    \end{align}
    Since statement (B) is assumed, there exists $C_{1} > 0$ such that
    \begin{equation}
        \langle S_{c}''(\phi_{c})w, w \rangle \geq C_{1} \| w \|_{H^{\sigma/2}}^{2}. \label{eq:B_006;BOst}
    \end{equation}
    Then, we obtain it from $M(u) = M(\phi_{c})$, \eqref{eq:B_001;BOst}, \eqref{eq:B_005;BOst}, \eqref{eq:B_006;BOst}, and \eqref{eq:B_004;BOst} that
    \begin{align}
        E(u)-E(\phi_{c}) = S_{c}(u) - S_{c}(\phi_{c}) & \geq \frac{1}{2} \langle S_{c}''(\phi_{c})v, v \rangle + o(\| v \|_{H^{\sigma/2}}^{2}) \\
        &= \frac{1}{2} \langle S_{c}''(\phi_{c})w, w \rangle + o(\| v \|_{H^{\sigma/2}}^{2}) \\
        &\geq \frac{C_{1}}{2} \| w \|_{H^{\sigma/2}}^{2} + o(\| v \|_{H^{\sigma/2}}^{2}) \\
        &\geq \frac{C_{1}}{8} \| v \|_{H^{\sigma/2}}^{2} + o(\| v \|_{H^{\sigma/2}}^{2}).
    \end{align}

    Finally, since $u \in U_{\varepsilon}(\phi_{c})$ and $v = u - \phi_{c}(\cdot - \tilde{y})$, we take $\varepsilon_{0} \in (0, \varepsilon_{1})$ so small that
    \begin{equation}
        E(u) - E(\phi_{c}) \geq \frac{C_{1}}{16} \| v \|_{H^{\sigma/2}}^{2}
    \end{equation}
    holds.
    Then, taking $C_{0} = C_{1}/16$ completes the proof.
\end{proof}

Due to the discussion above, we shall show that statement (B) holds so that travelling wave solutions are stable.
Actually, the statement (B) holds conditionally on $c > 0$ as follows:
\begin{prop} \label{prop:suff_cond_2;gkdv_st}
    \begin{enumerate}[label={\rm (\Roman*)} \ ]
        \item Assume condition {\rm (I)} in Theorem \ref{thm:stability;gkdv_st}.
        Let $\phi_{c}$ be a positive ground state solutions to \eqref{eq:Intro0300;gkdv_st} for $c > 0$.
        Then there exists $c_{0} \in (0, \infty)$ such that the statement (B) holds for all $c \in (0, c_{0})$.
        \item
        \begin{enumerate}[label={\rm (\arabic*)} \ ]
            \item Assume condition {\rm (II-1)} in Theorem \ref{thm:stability;gkdv_st}.
            Let $\phi_{c}$ be a positive ground state solution to \eqref{eq:Intro0300;gkdv_st} for $c > 0$.
            Then there exists $c_{1} \in (0, \infty)$ such that the statement (B) holds for all $c \in (c_{1}, \infty)$.

            \item Assume condition {\rm (II-2)} in Theorem \ref{thm:stability;gkdv_st}.
            Let $\phi_{c}$ be a negative ground state solutions to \eqref{eq:Intro0300;gkdv_st} for $c > 0$.
            Then there exists $c_{1} \in (0, \infty)$ such that the statement (B) holds for all $c \in (c_{2}, \infty)$.
        \end{enumerate}
    \end{enumerate}
\end{prop}

\section{Proof of Proposition \ref{prop:suff_cond_2;gkdv_st}} \label{section:suf_cond_2;gkdv_st}
In this section, we consider the proof of Proposition \ref{prop:suff_cond_2;gkdv_st}.
The method of the proof is inspired by Fukuizumi~\cite{Fukuizumi_2003}.

First, we give a formal observation of solutions to \eqref{eq:Intro0300;gkdv_st}.
Let $\phi \in H^{\sigma/2}(\mathbb{R})$ be a solution to \eqref{eq:Intro0300;gkdv_st}.
Then we put
\begin{equation}
    \phi(x) = c^{1/(p-1)} \tilde{\phi}(c^{1/\sigma}x) \label{eq:coercivity0100;gkdv_st}
\end{equation}
so that $\tilde{\phi}$ solves
\begin{equation}
    D_{x}^{\sigma} \tilde{\phi} + \tilde{\phi} -a \tilde{\phi}^{p} - c^{\alpha} \tilde{\phi}^{q} = 0, \quad x \in \mathbb{R}, \label{eq:coercivity0200;gkdv_st}
\end{equation}
where $\alpha = (q-p)/(p-1)$.
Similarly, putting
\begin{equation}
    \phi(x) = c^{1/(q-1)} \breve{\phi} (c^{1/\sigma}x), \label{eq:coercivity0300;gkdv_st}
\end{equation}
we see that $\breve{\phi}$ solves
\begin{equation}
    D_{x}^{\sigma} \breve{\phi} + \breve{\phi} -a c^{-\beta} \breve{\phi}^{p} - \breve{\phi}^{q} = 0, \quad x \in \mathbb{R}, \label{eq:coercivity0400;gkdv_st}
\end{equation}
where $\beta = (q-p)/(q-1)$.
Then, letting $c \rightarrow +0$ in \eqref{eq:coercivity0200;gkdv_st} or $c \rightarrow + \infty$ in \eqref{eq:coercivity0400;gkdv_st}, the following equation appears:
\begin{equation}
    D_{x}^{\sigma} \psi + \psi - \psi^{r} = 0, \quad x \in \mathbb{R}, \label{eq:coercivity0500;gkdv_st}
\end{equation}
where $r \in \{p, q\} $.
This equation has been observed well, and it is known that there exists a unique, positive and even ground state solution (up to translation) belonging to $H^{\infty}(\mathbb{R})$.
For details, see Frank--Lenzmann~\cite{Frank-Lenzmann}.

Here we observe properties of ground state solutions to \eqref{eq:coercivity0500;gkdv_st}.
We define the functional $S_{1}^{0,r}$ as
\begin{equation}
    S_{1}^{0,r}(v) \coloneq \frac{1}{2} \| v \|_{H^{\sigma/2}}^{2} - \frac{1}{r+1} \int_{\mathbb{R}} v^{r+1} \, dx, \quad v \in H^{\sigma/2}(\mathbb{R}),
\end{equation}
which is the action functional corresponding to \eqref{eq:coercivity0500;gkdv_st}.
Moreover, we define the Nehari functional $K_{1}^{0, r}$ derived from $S_{1}^{0, r}$ as
\begin{equation}
    K_{1}^{0, r}(v) \coloneq \langle (S_{1}^{0, r})'(v), v \rangle = \| v \|_{H^{\sigma/2}}^{2} - \int_{\mathbb{R}} v^{r+1} \, dx.
\end{equation}
Let $\psi_{1,r} \in H^{\sigma/2}(\mathbb{R})$ be the positive ground state solution to \eqref{eq:coercivity0500;gkdv_st}.
Then we can see it in a similar way to the discussion in Section \ref{subsection:GS_existence;gkdv_st} that the following characterization holds:
\begin{equation}
    S_{1}^{0, r}(\psi_{1, r}) = \inf\{ S_{1}^{0, r}(v): v \in H^{\sigma/2}(\mathbb{R}) \setminus \{0\}, \ K_{1}^{0,r}(v) = 0 \}.
\end{equation}
Suppose that $r$ is an odd integer and put $\chi_{1, r} \coloneq  - \psi_{1, r}$.
Then we see that $\chi_{1, r}$ is a negative ground state solution to \eqref{eq:coercivity0500;gkdv_st}.
Indeed, it is clear that $S_{1}^{0, r}(\chi_{1, r}) = S_{1}^{0, r}(\psi_{1, r})$ and $K_{1}^{0, r}(\chi_{1, r}) = 0$ hold.
Otherwise, if $r$ is an even integer, we see that
\begin{equation}
    K_{1}^{0, r}(w) = \| w \|_{H^{\sigma/2}}^{2} - \int_{\mathbb{R}} w^{r+1} \, dx = \| w \|_{H^{\sigma/2}}^{2} + \int_{\mathbb{R}} |w|^{r+1} \, dx \geq 0
\end{equation}
holds for any nonpositive function $w \in H^{\sigma/2}(\mathbb{R})$.
This means that there exist no negative solutions to \eqref{eq:coercivity0500;gkdv_st}.

One of the key lemmas to prove Proposition \ref{prop:suff_cond_2;gkdv_st} is the following.
\begin{lem} \label{lem:coercivity_single;gkdv_st}
    Let $1 \leq \sigma \leq 2$, $r \in \mathbb{N}$, and $2 \leq r < \infty$.
    Moreover, let $\psi_{1, r}$ be the positive ground state solution to \eqref{eq:coercivity0500;gkdv_st}.
    Then there exists $C_{2} > 0$ such that
    \begin{equation}
        \langle (S_{1}^{0,r})''(\psi_{1, r}) v, v \rangle \geq C_{2} \| v \|_{H^{\sigma/2}}^{2}
    \end{equation}
    holds for all $v \in H^{\sigma/2}(\mathbb{R})$ satisfying $(v, \psi_{1, r})_{L^{2}} = (v, \partial_{x} \psi_{1, r})_{L^{2}} = 0$.
\end{lem}

\begin{rem}
    If $r$ is odd, we can replace $\psi_{1, r}$ appearing in Lemma \ref{lem:coercivity_single;gkdv_st} with the negative ground state solution $\chi_{1, r}$.
\end{rem}

Next, we observe the convergence properties of ground state solutions to \eqref{eq:Intro0300;gkdv_st}.

\begin{lem} \label{lem:convergence;gkdv_st}
    \begin{enumerate}[label={\rm (\Roman*)} \ ]
        \item Assume condition {\rm (I)} in Theorem \ref{thm:existence_of_GS;gkdv_st}.
        Let $\phi_{c}$ be a positive ground state solution to \eqref{eq:Intro0300;gkdv_st} for $c > 0$, and $\tilde{\phi}_{c}$ be a function given by the scaling \eqref{eq:coercivity0100;gkdv_st}.
        Moreover, let $\psi_{1, p}$ be the positive ground state solution to \eqref{eq:coercivity0500;gkdv_st} with $r = p$.
        Then it holds that $\tilde{\phi}_{c} \rightarrow \psi_{1, p}$ strongly in $H^{\sigma/2}(\mathbb{R})$ as $c \rightarrow +0$.

        \item
        \begin{enumerate}[label={\rm (\arabic*)} \ ]
            \item Assume condition {\rm (II-1)} in Theorem \ref{thm:existence_of_GS;gkdv_st}.
            Let $\phi_{c}$ be a positive ground state solution to \eqref{eq:Intro0300;gkdv_st} for $c > 0$, and $\breve{\phi}_{c}$ be a function given by the scaling \eqref{eq:coercivity0300;gkdv_st}.
            Moreover, let $\psi_{1, q}$ be the positive ground state solution to \eqref{eq:coercivity0500;gkdv_st} with $r = q$.
            Then it holds that $\breve{\phi}_{c} \rightarrow \psi_{1, q}$ strongly in $H^{\sigma/2}(\mathbb{R})$ as $c \rightarrow +\infty$.

            \item Assume condition {\rm (II-2)} in Theorem \ref{thm:existence_of_GS;gkdv_st}.
            Let $\phi_{c}$ be a negative ground state solution to \eqref{eq:Intro0300;gkdv_st} for $c > 0$, and $\breve{\phi}_{c}$ be a function given by the scaling \eqref{eq:coercivity0300;gkdv_st}.
            Moreover, let $\psi_{1, q}$ be the positive ground state solution to \eqref{eq:coercivity0500;gkdv_st} with $r = q$, and $\chi_{1, q} \coloneq - \psi_{1, q}$.
            Then, it holds that $\breve{\phi}_{c} \rightarrow \chi_{1, q}$ strongly in $H^{\sigma/2}(\mathbb{R})$ as $c \rightarrow +\infty$.
        \end{enumerate}
    \end{enumerate}
\end{lem}

We will give the proofs of Lemmas \ref{lem:coercivity_single;gkdv_st} and \ref{lem:convergence;gkdv_st} later in this section.
Here we prove Proposition \ref{prop:suff_cond_2;gkdv_st} by applying these lemmas.

\begin{proof}[Proof of Proposition \ref{prop:suff_cond_2;gkdv_st}]
    First, We consider case (I).

    We define the action functional $\tilde{S}_{c}$ corresponding to \eqref{eq:coercivity0200;gkdv_st} as
    \begin{equation}
        \tilde{S}_{c}(v) \coloneq \frac{1}{2} \| v \|_{H^{\sigma/2}}^{2} - \frac{1}{p+1} \int_{\mathbb{R}} v^{p+1} \, dx - \frac{c^{\alpha}}{q+1} \int_{\mathbb{R}} v^{q+1} \, dx, \quad v \in H^{\sigma/2}(\mathbb{R}).
    \end{equation}

    Moreover, for $v \in H^{\sigma/2}(\mathbb{R})$ and $c > 0$, we put
    \begin{align}
        \langle L_{c}v,v \rangle &\coloneq \langle S_{c}''(\phi_{c})v, v \rangle = \| v \|_{H^{\sigma/2}_{c}}^{2} - p\int_{\mathbb{R}} \phi_{c}^{p-1}v^{2} \, dx - q \int_{\mathbb{R}} \phi_{c}^{q-1}v^{2} \, dx, \\
        \langle \tilde{L}_{c}v ,v \rangle &\coloneq \langle \tilde{S}_{c}''(\tilde{\phi}_{c})v, v \rangle = \| v \|_{H^{\sigma/2}}^{2} - p \int_{\mathbb{R}} \tilde{\phi}_{c}^{p-1} v^{2} \, dx - c^{\alpha} q \int_{\mathbb{R}} \tilde{\phi}_{c}^{q-1}v^{2} \, dx, \\
        \langle L_{1}^{0,p}v, v \rangle &\coloneq \langle (S_{1}^{0,p})''(\psi_{1, p})v, v \rangle = \| v \|_{H^{\sigma/2}}^{2} - p \int_{\mathbb{R}} \psi_{1, p}^{p-1} v^{2} \, dx.
    \end{align}
    Then we can see that $\langle L_{c}v, v \rangle = c^{1 + 2/(p-1) - 1/\sigma} \langle \tilde{L}_{c} \tilde{v}, \tilde{v} \rangle$ with the scaling $v(x) = c^{1/(p-1)} \tilde{v}(c^{1/\sigma} x)$.

    Now we prove this proposition with contradiction.
    Suppose that the statement is failed.
    Then we can take sequences $(c_{n})_{n} \subset (0, \infty)$ and $(v_{n})_{n} \subset H^{\sigma/2}(\mathbb{R})$ which satisfy the following conditions:
    \begin{align}
        & c_{n} \rightarrow 0 \ \text{as} \ n \rightarrow \infty; \label{eq:suffcond2_0100;gkdv_st} \\
        & \| v_{n} \|_{L^{2}} = 1 \ \text{for all} \ n \in \mathbb{R}; \label{eq:suffcond2_0200;gkdv_st} \\
        & \limsup_{n \rightarrow \infty} \langle \tilde{L}_{c_{n}}v_{n}, v_{n} \rangle \leq 0; \label{eq:suffcond2_0300;gkdv_st} \\
        & (v_{n}, \tilde{\phi}_{c_{n}})_{L^{2}} = (v_{n}, \partial_{x} \tilde{\phi}_{c_{n}})_{L^{2}} = 0 \ \text{for all} \ n \in \mathbb{R}. \label{eq:suffcond2_0400;gkdv_st}
    \end{align}
    By Lemma \ref{lem:convergence;gkdv_st}, we see that
    \begin{equation}
        \tilde{\phi}_{c_{n}} \rightarrow \psi_{1, p} \ \text{in} \ H^{\sigma/2}(\mathbb{R}). \label{eq:suffcond2_0600;gkdv_st}
    \end{equation}
    Additionally, by \eqref{eq:suffcond2_0200;gkdv_st}, we can see that
    \begin{equation}
        v_{n} \rightharpoonup v_{0} \ \text{weakly in} \ H^{\sigma/2}(\mathbb{R}) \label{eq:suffcond2_0700;gkdv_st}
    \end{equation}
    with some $v_{0} \in H^{\sigma/2}(\mathbb{R})$, up to a subsequence.
    By \eqref{eq:suffcond2_0400;gkdv_st}, \eqref{eq:suffcond2_0600;gkdv_st}, and \eqref{eq:suffcond2_0700;gkdv_st}, we obtain $(v_{0}, \psi_{1, r})_{L^{2}} = 0$.
    Moreover, we can see that
    \begin{equation}
        \partial_{x} \tilde{\phi}_{c_{n}} \rightarrow \partial_{x} \psi_{1, p} \ \text{strongly in} \ H^{-\sigma/2}(\mathbb{R}). \label{eq:suffcond2_0810;gkdv_st}
    \end{equation}
    Indeed, since $\sigma \geq 1$, we have
    \begin{align}
        \| \partial_{x} \tilde{\phi}_{c_{n}} - \partial_{x} \psi_{1, p} \|_{H^{-\sigma/2}} &\leq C \| \langle \xi \rangle^{-\sigma/2} \mathscr{F}[ \partial_{x} \tilde{\phi}_{c_{n}} - \partial_{x} \psi_{1, p} ] \|_{L^{2}} \\
        &\leq C \| \langle \xi \rangle^{1- \sigma/2} \mathscr{F}[\tilde{\phi}_{c_{n}} - \psi_{1, p}] \|_{L^{2}} \\
        &\leq C \| \langle \xi \rangle^{\sigma/2} \mathscr{F}[\tilde{\phi}_{c_{n}} - \psi_{1, p}] \|_{L^{2}} \\
        &\leq C \| \tilde{\phi}_{c_{n}} - \psi_{1, p} \|_{H^{\sigma/2}}
    \end{align}
    with some constant $C > 0$ independent of $n$.
    This estimate and \eqref{eq:suffcond2_0600;gkdv_st} imply \eqref{eq:suffcond2_0810;gkdv_st}.
    By \eqref{eq:suffcond2_0810;gkdv_st}, we have
    \begin{equation}
        (v_{0}, \partial_{x} \psi_{1, p})_{L^{2}} = \langle \partial_{x} \psi_{1, p}, v_{0} \rangle = \lim_{n \rightarrow \infty} \langle \partial_{x} \tilde{\phi}_{c_{n}}, v_{n} \rangle = \lim_{n \rightarrow \infty} (v_{n}, \partial_{x} \tilde{\phi}_{c_{n}})_{L^{2}} = 0.
    \end{equation}
    Therefore, we can apply Lemma \ref{lem:coercivity_single;gkdv_st} to $v_{0}$ so that we have
    \begin{equation}
        C_{2} \| v_{0} \|_{H^{\sigma/2}}^{2} \leq \langle L_{1}^{0,p} v_{0}, v_{0} \rangle. \label{eq:suffcond2_0820;gkdv_st}
    \end{equation}

    Furthermore, by \eqref{eq:suffcond2_0600;gkdv_st} and \eqref{eq:suffcond2_0700;gkdv_st}, we see that
    \begin{equation}
        \int_{\mathbb{R}} \tilde{\phi}_{c_{n}}^{\gamma-1} v_{n}^{2} \, dx \rightarrow \int_{\mathbb{R}} \psi_{1, p}^{\gamma-1}v_{0}^{2} \, dx \label{eq:suffcond2_0900;gkdv_st}
    \end{equation}
    for $\gamma \in (1, \infty)$.
    Combining \eqref{eq:suffcond2_0300;gkdv_st}, and \eqref{eq:suffcond2_0900;gkdv_st}, we obtain
    \begin{equation}
        \langle L_{1}^{0,p} v_{0}, v_{0} \rangle = \| v_{0} \|_{H^{\sigma/2}}^{2} - p \int_{\mathbb{R}} \psi_{1, p}^{p-1}v_{0}^{2} \, dx \leq \liminf_{n \rightarrow \infty} \langle \tilde{L}_{c_{n}} v_{n}, v_{n} \rangle \leq 0. \label{eq:suffcond2_1000;gkdv_st}
    \end{equation}
    Then \eqref{eq:suffcond2_0820;gkdv_st} and \eqref{eq:suffcond2_1000;gkdv_st} yield that $v_{0} = 0$.

    However, \eqref{eq:suffcond2_0200;gkdv_st} gives
    \begin{equation}
        \| v_{0} \|_{H^{\sigma/2}} \leq \liminf_{n \rightarrow \infty} \| v_{n} \|_{H^{\sigma/2}} = 1. \label{eq:suffcond2_1100;gkdv_st}
    \end{equation}
    Then \eqref{eq:suffcond2_0300;gkdv_st}, \eqref{eq:suffcond2_0900;gkdv_st}, and \eqref{eq:suffcond2_1100;gkdv_st} imply that
    \begin{equation}
        0 \geq \liminf_{n \rightarrow \infty} \langle \tilde{L}_{c_{n}}v_{n}, v_{n} \rangle \geq 1 - p \int_{\mathbb{R}} \psi_{1, p}^{p-1}v_{0}^{2} \, dx = 1,
    \end{equation}
    which is impossible.
    Hence, the proof for case (I) is accomplished.

    In case (II-$j$) ($j = 1, 2$), we put
    \begin{equation}
        \langle \breve{L}_{c}v, v \rangle \coloneq \langle \breve{S}_{c}''(\breve{\phi}_{c})v, v \rangle = \| v \|_{H^{\sigma/2}}^{2} - a c^{-\beta} p \int_{\mathbb{R}} \breve{\phi}_{c}^{p-1} v^{2}\, dx - q \int_{\mathbb{R}} \breve{\phi}_{c}^{q-1} v^{2} \, dx,
    \end{equation}
    so that we see that $\langle L_{c}v, v \rangle = c^{1 + 2/(q-1) - 1/\sigma} \langle \breve{L}_{c} \breve{v}, \breve{v} \rangle$ holds for $v \in H^{\sigma/2}(\mathbb{R})$, where $\breve{S}_{c}$ is the action functional corresponding to stationary problem \eqref{eq:coercivity0400;gkdv_st} defined as
    \begin{equation}
        \breve{S}_{c}(v) \coloneq \frac{1}{2} \| v \|_{H^{\sigma/2}}^{2} + \frac{c^{- \beta}}{p+1} \| v \|_{L^{p+1}}^{p+1} - \frac{1}{q+1} \int_{\mathbb{R}} v^{q+1} \, dx,
    \end{equation}
    and $\breve{v}$ is a function given by the scaling $v(x) = c^{1/(q-1)} \breve{v}(c^{1/\sigma} x)$.
    Then we can prove Proposition \ref{prop:suff_cond_2;gkdv_st} for case (II-$j$) similarly to case (I).
\end{proof}

\subsection{Proof of Lemma \ref{lem:coercivity_single;gkdv_st}}

First, we observe the regularity of ground state solutions \eqref{eq:coercivity0500;gkdv_st}.

\begin{lem} \label{lem:suffcond2_regularity_energysp;gkdv_st}
    Let $\psi_{1, r} \in H^{\sigma/2}(\mathbb{R})$ be the positive ground state solution to \eqref{eq:coercivity0500;gkdv_st}.
    Then $x \partial_{x}\psi_{1, r} \in H^{\sigma/2}(\mathbb{R})$ holds.
\end{lem}

Here we consider the case that $1 < \sigma \leq 2$.
\begin{lem} \label{lem:AppA_regularity;BOst}
    Let $1 < \sigma \leq 2$ and $f\in H^{-\sigma/2}(\mathbb{R})$.
    Moreover, assume that $\psi \in H^{\sigma + 1}(\mathbb{R})$ satisfies
    \begin{equation}
        D_{x}^{\sigma} \psi + \psi = f \quad \text{in} \ H^{-\sigma/2}(\mathbb{R}). \label{eq:AppA_002;BOst}
    \end{equation}
    If $xf \in H^{1-\sigma/2}(\mathbb{R})$, then $x \partial_{x}\psi \in H^{\sigma/2}(\mathbb{R})$ holds.
\end{lem}
\begin{proof}
    In this proof, we let $\hat{u}$ denote the Fourier transform of $u$.
    Here we remark that there exist some $C > 0$ such that $C \langle \xi \rangle^{\sigma} \leq 1 + |\xi|^{\sigma} \leq C \langle \xi \rangle^{\sigma}$ holds for $\xi \in \mathbb{R}$, where $\langle \xi \rangle \coloneq (1 + |\xi|^{2})^{1/2}$ for $\xi \in \mathbb{R}$.

    First, rewriting \eqref{eq:AppA_002;BOst}, we have
    \begin{equation}
        \hat{\psi} = \frac{1}{1+|\xi|^{\sigma}} \hat{f}. \label{eq:AppA_003;BOst}
    \end{equation}
    Then a direct calculation yields
    \begin{equation}
        \mathscr{F}[x \partial_{x}\psi] = i \partial_{\xi} \mathscr{F} [\partial_{x}\psi] = - \partial_{\xi}\{ \xi \hat{\psi} \} = - \hat{\psi} - \xi \partial_{\xi}\hat{\psi}. \label{eq:AppA_004;BOst}
    \end{equation}
    Moreover, we have
    \begin{align}
        \partial_{\xi} \bigg\{ \frac{1}{1 + |\xi|^{\sigma}} \hat{f} \bigg\} = -\frac{\sigma |\xi|^{\sigma-2}\xi}{(1+|\xi|^{\sigma})^{2}} \hat{f} + \frac{1}{1 + |\xi|^{\sigma}} \partial_{\xi}\hat{f}. \label{eq:AppA_005;BOst}
    \end{align}
    Additionally, it is clear that
    \begin{equation}
        \mathscr{F}[xf] = i \partial_{\xi} \hat{f}. \label{eq:AppA_006;BOst}
    \end{equation}
    Then, by \eqref{eq:AppA_003;BOst}, \eqref{eq:AppA_004;BOst}, \eqref{eq:AppA_005;BOst}, and \eqref{eq:AppA_006;BOst}, we have
    \begin{equation}
        \left| \langle \xi \rangle^{\sigma/2} \mathscr{F}[x \partial_{x}\psi] \right| \leq C \left( \langle \xi \rangle^{\sigma/2} |\hat{\psi}| + \langle \xi \rangle^{-\sigma/2} |\hat{f}| + \langle \xi \rangle^{1 - \sigma/2} |\mathscr{F}[xf]| \right)
    \end{equation}
    Since we assume that $f \in H^{-\sigma/2}(\mathbb{R})$ and $xf \in H^{1 - \sigma/2}(\mathbb{R})$, we obtain $ \langle \xi \rangle^{\sigma/2} \mathscr{F}[x \partial_{x}\psi] \in L^{2}(\mathbb{R})$, which means $x \partial_{x}\psi \in H^{\sigma/2}(\mathbb{R})$.
\end{proof}

Now we prove Lemma \ref{lem:suffcond2_regularity_energysp;gkdv_st} for $1 < \sigma \leq 2$.
\begin{proof}
    According to Lemma \ref{lem:AppA_regularity;BOst}, it suffices to show that $x \psi_{1, r}^{r} \in H^{1-\sigma/2}(\mathbb{R})$.

    When $\sigma = 2$, it is known that $\psi_{1, r}$ and $x \partial_{x} \psi_{1, r}$ decay exponentially as $|x| \rightarrow + \infty$.
    Then we can see that $x \psi_{1, r}^{r} \in H^{1-\sigma/2}(\mathbb{R})$.

    Now we consider the case $1 < \sigma < 2$.
    First, we remark that, for $1 \leq \sigma < 2$, it holds that
    \begin{equation}
        |\psi_{1, r}(x)| + |x \partial_{x} \psi_{1, r}(x)| \leq \frac{C}{1 + |x|^{1 + \sigma}} \label{eq:suffcond2_3599;gkdv_st}
    \end{equation}
    for all $x \in \mathbb{R}$ with some constant $C > 0$, which is observed by Kenig--Martel--Robibiano~\cite{Kenig-Martel-Robbiano} and Frank--Lenzmann~\cite{Frank-Lenzmann} with the method of Amick--Toland~\cite{Amick-Toland_1}.
    
    Similarly to the proof of Lemma \ref{lem:GS_regularlity_of_sol;gkdv_st}, we can see that $\psi_{1, r}^{r} \in H^{1}(\mathbb{R})$, which implies that $\psi_{1, r}^{r} \in H^{-\sigma/2}(\mathbb{R})$.

    Moreover, we can see that $|x| |\psi_{1, r}(x)|^{r} = O(|x|^{-(\sigma+1)r + 1})$ as $|x| \rightarrow + \infty$.
    Form this, it is sufficient for $x \psi_{1, r}^{r}$ to be in $L^{2}(\mathbb{R})$ that
    \begin{equation}
        -(\sigma+1)r + 1 < - \frac{1}{2} \Longleftrightarrow r > \frac{3}{2(\sigma+1)}, \label{eq:AppA_009a;BOst}
    \end{equation}
    which holds true under $r \geq 2$ and $1 \leq \sigma \leq 2$.

    Next, since $x \psi_{1, r}^{r} \in C^{1}(\mathbb{R})$, we see that
    \begin{equation}
        \left| \partial_{x} \left\{ x \psi_{1, r}^{r}(x) \right\} \right| \leq |\psi_{1,r}(x)|^{r} + r |\psi_{1, r}(x)|^{r-1}|x \partial_{x} (\psi_{1, r})(x)| = O(|x|^{-(\sigma + 1)r})
    \end{equation}
    as $|x| \rightarrow + \infty$.
    Then, similarly above, we obtain  $\partial_{x} \left\{ x \psi_{1, r}^{r} \right\} \in L^{2}(\mathbb{R})$.
    Hence, we see that $x \psi_{1, r}^{r} \in H^{1}(\mathbb{R}) \hookrightarrow H^{1 - \sigma/2}(\mathbb{R}) $ and that the statement follows from Lemma \ref{lem:AppA_regularity;BOst}.
\end{proof}

We cannot prove Lemma \ref{lem:suffcond2_regularity_energysp;gkdv_st} for $\sigma = 1$ in the same way as the case $1 < \sigma \leq 2$ since the differentiability of \eqref{eq:AppA_004;BOst} is failed.
For the case $\sigma = 1$, we observe decay estimate of the second derivative of the ground state solution $\psi_{1, r}$ with the idea of Amick--Toland~\cite{Amick-Toland_1}.

Differentiating $x \partial_{x}\psi_{1, r}$ with respect to $x$, we obtain
\begin{equation}
    \partial_{x}\{ x \partial_{x}\psi_{1, r} \} = \partial_{x} \psi_{1,r} + x \partial_{x}^{2} \psi_{1, r}.
\end{equation}
Therefore, in order to prove Lemma \ref{lem:suffcond2_regularity_energysp;gkdv_st}, it suffices to show that $x \partial_{x}\psi_{1, r} \in L^{2}(\mathbb{R})$ and $x \partial^{2}\psi_{1, r}(\mathbb{R}) \in L^{2}(\mathbb{R})$.
The first integrability follows from the decay estimate \eqref{eq:suffcond2_3599;gkdv_st}.
Now we prove the second integrability.

\begin{lem} \label{lem:suffcond2_decay_2nd;gkdv_st}
    There exists some $C > 0$ such that
    \begin{equation}
        |\partial_{x}^{2} \psi_{1, r}(x)| \leq \frac{C}{|x|^{1 + \sigma}} \label{eq:suffcond2_6000;gkdv_st}
    \end{equation}
    holds for all $|x| > 1$.
\end{lem}
\begin{proof}
    We consider the following function:
    \begin{equation}
        G_{1}(x, y) \coloneq \frac{1}{\pi} \int_{0}^{\infty} \frac{s e^{-s}}{x^{2} + (y+s)^{2}} \/ ds
    \end{equation}
    for $x \in \mathbb{R}$ and $y \geq 0$, which is introduced by Amick--Toland~\cite{Amick-Toland_1,Amick-Toland_2}.
    We can see that the function $G_{1}$ is harmonic in $\{ (x, y) \in \mathbb{R}^{2}: y > 0 \}$ and that $\int_{\mathbb{R}} G_{1}(x, y) \, dx = 1$ holds for any $y \geq 0$.
    Moreover, for $|x| > 1/2$ and $y \geq 0$, we can see that
    \begin{align}
        & |\partial_{x} G_{1}(x, y)| \leq \frac{C}{|x|}G_{1}(x, y), \label{eq:suffcond2_6500;gkdv_st} \\
        & |\partial_{x}^{2} G_{1}(x, y)| \leq \frac{C}{x^{2}}G_{1}(x, y). \label{eq:suffcond2_6600;gkdv_st}
    \end{align}
    
    As a consequence of Amick--Toland's discussion in \cite{Amick-Toland_1,Amick-Toland_2}, we can write the positive ground state solution $\psi_{1, r}$ as
    \begin{equation}
        \psi_{1, r}(x) = (G_{1}(\cdot, 0) \ast \psi_{1, r}^{r})(x) = \int_{-\infty}^{\infty} G_{1}(x-t, 0) \psi_{1, r}^{r}(t) \, dt.
    \end{equation}
    Since $\psi_{1, r} \in C^{\infty}(\mathbb{R})$, we obtain
    \begin{equation}
        \partial_{x}^{2} \psi_{1, r}(x) = \int_{-\infty}^{\infty} G_{1}(x-t, 0) \partial_{t}^{2} \{\psi_{1, r}^{r}(t)\} \, dt. \label{eq:suffcond2_6100;gkdv_st}
    \end{equation}
    Let $x > 1$.
    Then we split the right hand side of \eqref{eq:suffcond2_6100;gkdv_st} as
    \begin{equation}
        \partial_{x}^{2} \psi_{1, r}(x) = \int_{-\infty}^{x/2} G_{1}(x-t, 0) \partial_{t}^{2} \{ \psi_{1, r}^{r}(t) \} \, dt + \int_{x/2}^{\infty} G_{1}(x-t, 0) \partial_{t}^{2} \{ \psi_{1, r}^{r}(t) \} \, dt =: I_{1}(x) + I_{2}(x).
    \end{equation}
    Now we evaluate $I_{1}$.
    Doing integration by parts twice, we obtain
    \begin{equation}
        I_{1}(x) = r G_{1}\left( \frac{x}{2}, 0 \right) \psi_{1, r}^{r-1}\left( \frac{x}{2} \right) \partial_{x}\psi_{1, r} \left( \frac{x}{2} \right) + \partial_{x}G_{1}\left( \frac{x}{2}, 0 \right) \psi_{1, r}^{r}\left( \frac{x}{2} \right) + \int_{- \infty}^{x/2} \partial_{x}^{2} G_{1}(x-t, 0) \psi_{1, r}^{r}(t) \, dt.
    \end{equation}
    Then, by \eqref{eq:suffcond2_3599;gkdv_st}, \eqref{eq:suffcond2_6500;gkdv_st}, and \eqref{eq:suffcond2_6600;gkdv_st}, we have
    \begin{align}
        |I_{1}(x)| &\leq C \left( x^{-2} x^{-2(r-1)} x^{-3} + x^{-3} x^{-2r} + \int_{-\infty}^{x/2} \frac{G_{1}(x-t, 0)}{(x-t)^{2}} \psi_{1, r}^{r}(t) \, dt  \right) \\
        &\leq C \left( x^{-(2r + 3)} + x^{-2} \psi_{1, r}(x) \right) \\
        &\leq C x^{-4}. \label{eq:suffcond2_6200;gkdv_st}
    \end{align}
    Next, we consider $I_{2}$.
    A direct calculation yields
    \begin{equation}
        |I_{2}(x)| \leq r(r-1) \left| \int_{x/2}^{\infty} G_{1}(x-t, 0) \psi_{1, r}^{r-2}(t) (\partial_{x} \psi_{1, r}(t))^{2} \, dt \right| + r \left| \int_{x/2}^{\infty} G_{1}(x-t, 0) \psi_{1, r}^{p-1}(t) \partial_{x}^{2} \psi_{1, r}(t) \, dt \right|.
    \end{equation}
    By the decay estimate \eqref{eq:suffcond2_3599;gkdv_st}, we have
    \begin{equation}
        \left| \int_{x/2}^{\infty} G_{1}(x-t, 0) \psi_{1, r}^{r-2}(t) (\partial_{x} \psi_{1, r}(t))^{2} \, dt \right| \leq C x^{-2(r-1)} x^{-6} \int_{-\infty}^{\infty} G_{1}(x-t, 0) \, dt = Cx^{-(2r-4)}.
    \end{equation}
    Moreover, we see that
    \begin{align}
        \left| \int_{x/2}^{\infty} G_{1}(x-t, 0) \psi_{1, r}^{p-1}(t) \partial_{x}^{2} \psi_{1, r}(t) \, dt \right| &\leq C x^{-2(r-1)} \| \partial_{x}^{2} \psi_{1, r} \|_{L^{\infty}} \int_{-\infty}^{\infty} G_{1}(x-t, 0) \, dx \\
        &\leq C x^{-2(r-1)}.
    \end{align}
    Therefore, we obtain that
    \begin{equation}
        |I_{2}(x)| \leq C x^{-2(r-1)}. \label{eq:suffcond2_6300;gkdv_st}
    \end{equation}
    Finally, \eqref{eq:suffcond2_6200;gkdv_st} and \eqref{eq:suffcond2_6300;gkdv_st} yield that
    \begin{equation}
        |\partial_{x}^{2} \psi_{1,r}(x)| \leq C x^{-2(r-1)} \leq C x^{-2}
    \end{equation}
    holds for $x > 1$.

    Since $\psi_{1, r}$ is an even function, we conclude the desired estimate.
\end{proof}

By Lemma \ref{lem:suffcond2_decay_2nd;gkdv_st}, it holds for $|x| > 1$ that
\begin{equation}
    |x \partial_{x}^{2} \psi_{1, r}(x)| \leq C |x|^{-1},
\end{equation}
which implies that $x \partial_{x}^{2} \psi_{1, r} \in L^{2}(\mathbb{R})$.
Hence, the proof of Lemma \ref{lem:suffcond2_regularity_energysp;gkdv_st} for $\sigma = 1$ is completed.

Now we proceed the proof of Lemma \ref{lem:coercivity_single;gkdv_st}.
First, we claim an abstract lemma.
\begin{lem} \label{lem:suffcond2_abstract_statement;gkdv_st}
    Let $X$ be a Hilbert space over $\mathbb{R}$, $S, K \in C^{2}(X, \mathbb{R})$, and $f \in X^{\ast}$, where $X^{\ast}$ is the dual space of $X$.
    Let $\psi \in X \setminus \{0\}$ satisfy $K(\psi) = 0$ and
    \begin{equation}
        S(\psi) = \min \{ S(v): v \in X \setminus \{0\}, \ K(v) = 0 \}.
    \end{equation}
    Moreover, assume the following five conditions:
    \begin{enumerate}[label={\rm (C\arabic*)}]
        \item There exists $v_{1} \in X$ such that \label{cond:C1;gkdv_st}
        \begin{equation}
            \langle S'(\psi), v_{1} \rangle = 0, \quad \langle K(\psi), v_{1} \rangle \neq 0.
        \end{equation}

        \item There exists $v_{2} \in X$ such that \label{cond:C2;gkdv_st}
        \begin{equation}
            S''(\psi)v_{2} = -f, \quad \langle K'(\psi), v_{2} \rangle \neq 0, \quad \langle f, v_{2} \rangle > 0.
        \end{equation}

        \item The functional $X \ni v \mapsto \langle S''(\psi)v, v \rangle \in \mathbb{R}$ is weakly lower semicontiuous. \label{cond:C3;gkdv_st}

        \item Let $(v_{n})_{n} \subset X$ satisfy that $\| v_{n} \|_{X} = 1$ for all $n \in \mathbb{R}$ and that $v_{n} \rightharpoonup 0$ weakly in $X$.
        Then it holds that \label{cond:C4;gkdv_st}
        \begin{equation}
            \liminf_{n \rightarrow \infty} \langle S''(\psi)v_{n}, v_{n} \rangle > 0.
        \end{equation}

        \item $\langle f, v \rangle = 0$ holds for all $v \in \operatorname{ker} S''(\psi)$. \label{cond:C5;gkdv_st}

    \end{enumerate}

    Then there exists $C_{3} > 0$ such that
    \begin{equation}
        \langle S''(\psi)v, v \rangle \geq C_{3} \| v \|_{X}
    \end{equation}
    holds for all $v \in W$ satisfying $\langle f, v \rangle = 0$, where $W$ is a closed subspace of $X$ satisfying that $W \cap \operatorname{ker} S''(\psi) = \{0\}$.
\end{lem}

Now we prove Lemma \ref{lem:suffcond2_abstract_statement;gkdv_st}.
Hereafter, we let $\langle \cdot, \cdot \rangle$ denote a dual product between $X^{\ast}$ and $X$ until the proof is accomplished.

\begin{lem} \label{lem:AppB_positiveness_nehari;BOst}
    Let $S, K \in C^{2}(X, \mathbb{R})$ and assume \ref{cond:C1;gkdv_st}.
    Then $\langle S''(\psi)v, v \rangle \geq 0$ holds for all $v \in X$ satisfying $\langle K'(\psi), v \rangle = 0$.
\end{lem}
\begin{proof}
    Let $\Omega \subset \mathbb{R}^{2}$ be a suitable neighborhood of $(0,0) \in \mathbb{R}^{2}$ and put
    \begin{align}
        f(s, t) &\coloneqq S(\psi + sv + tv_{1}) \\
        g(s, t) &\coloneqq K(\psi + sv + tv_{1})
    \end{align}
    for $(s,t) \in \Omega$.
    We can see that $f, g \in C^{2}(\Omega)$ and that
    \begin{equation}
        g(0,0) = K(\psi) = 0, \quad \partial_{t}g(0,0) = \langle K'(\psi), v_{1} \rangle \neq 0.
    \end{equation}
    Then, by applying the implicit function theorem, there exists $\gamma \in C^{2}(-\delta, \delta)$ which satisfies $\gamma(0) = 0$ and
    $g(s, \gamma(s)) = 0$ for $s \in (-\delta, \delta)$ with some constant $\delta > 0$.
    Therefore, we can obtain
    \begin{equation}
        \partial_{s} g(0,0) + \partial_{t} g(0,0) \gamma'(0) = 0,
    \end{equation}
    which gives
    \begin{equation}
        \gamma'(0) = - \frac{\partial_{s} g(0,0)}{\partial_{t} g(0,0)} = -\frac{\langle K'(\psi), v \rangle}{\langle K'(\psi), v_{1} \rangle} = 0.
    \end{equation}
    Here we put $h(s) \coloneqq f(s, \gamma(s)) = S(\psi + sv + \gamma v_{1})$.
    Since
    $\psi + sv + \gamma(s)v_{1} \neq 0$ and $K(\psi + sv + \gamma(s)v_{1})  = 0$ hold for $s \in (-\delta, \delta)$, we see that $h$ attains its local minimum at $s = 0$.
    Therefore, we have
    \begin{align}
        0 \leq h''(0) &= \partial_{s}^{2} f(0,0) + 2 \partial_{s} \partial_{t} f(0,0) \gamma'(0) + \partial_{t}^{2} f(0,0) \gamma'(0)^{2} + \partial_{t}f(0,0) \gamma''(0) \\
        &= \langle S''(\psi)v, v \rangle + \gamma''(0) \langle S'(\psi), v_{1} \rangle = \langle S''(\psi)v, v \rangle.
    \end{align}
    This completes the proof.
\end{proof}

\begin{lem} \label{lem:AppB_positiveness_f;BOst}
    Let $S, K \in C^{2}(X, \mathbb{R})$ and assume \ref{cond:C1;gkdv_st} and \ref{cond:C2;gkdv_st}.
    Then $\langle S''(\psi)v, v \rangle \geq 0$ holds for all $v \in X$ satisfying $\langle f, v \rangle = 0$
\end{lem}
\begin{proof}
    Let $v \in X$ satisfy $\langle f, v \rangle = 0$ and set
    \begin{equation}
        \kappa \coloneq - \frac{\langle K'(\psi), v \rangle}{\langle K'(\psi), v_{2} \rangle}
    \end{equation}
    so that $\langle K'(\psi), v + \kappa v_{2} \rangle = 0$ holds.
    Then, by Lemma \ref{lem:AppB_positiveness_nehari;BOst} and \ref{cond:C2;gkdv_st}, we have
    \begin{align}
        0 &\leq \langle S''(\psi)(v + \kappa v_{2}), v + \kappa v_{2} \rangle \\
        &= \langle S''(\psi)v, v \rangle + 2\kappa \langle S''(\psi)v_{2}, v \rangle + \kappa^{2} \langle S''(\psi)v_{2}, v_{2} \rangle \\
        &= \langle S''(\psi)v, v \rangle - 2\kappa \langle f, v \rangle - \kappa^{2} \langle f, v_{2} \rangle \\
        &\leq \langle S''(\psi)v, v \rangle.
    \end{align}
    This completes the proof.
\end{proof}

Now we proceed the proof of Lemma \ref{lem:suffcond2_abstract_statement;gkdv_st}
\begin{proof}[Proof of Lemma \ref{lem:suffcond2_abstract_statement;gkdv_st}]
    We prove this lemma with contradiction.

    Suppose that the statement is failed.
    Then there exists a sequence $(v_{n})_{n} \subset X$ which satisfies that $\| v_{n} \|_{X} = 1, \ v_{n} \in W$, and $\langle f, v_{n} \rangle = 0$ for all $n \in \mathbb{N}$, and that
    \begin{equation}
        \limsup_{n \rightarrow \infty} \langle S''(\psi)v_{n}, v_{n} \rangle \leq 0. \label{eq:AppB_001;BOst}
    \end{equation}
    Since $(v_{n})_{n}$ is bounded in $X$ and $W$ is closed, up to a subsequence, we obtain $v_{n} \rightharpoonup v_{0}$ weakly in $X$ with some $v \in W$, which gives $\langle f, v_{0} \rangle = 0$.
    By \ref{cond:C3;gkdv_st} and \eqref{eq:AppB_001;BOst}, we have
    \begin{equation}
        0 \leq \langle S''(\psi)v_{0}, v_{0} \rangle \leq \liminf_{n \rightarrow \infty} \langle S''(\psi)v_{n}, v_{n} \rangle \leq 0,
    \end{equation}
    that is, $\langle S''(\psi)v_{0}, v_{0} \rangle = 0$.
    Moreover, we can see that $v_{0} \neq  0$ from \ref{cond:C4;gkdv_st}.
    Therefore, by Lemma \ref{lem:AppB_positiveness_f;BOst}, we obtain
    \begin{equation}
        \langle S''(\psi)v_{0}, v_{0} \rangle = 0 = \min \{ \langle S''(\psi)v, v \rangle : v \in X \setminus \{ 0 \}, \langle f, v \rangle = 0 \}.
    \end{equation}
    Then, there exists a Lagrange multiplier $\mu \in \mathbb{R}$ such that $S''(\psi)v_{0} = \mu f$.
    Using $S''(\psi)v_{2} = -f$ in \ref{cond:C2;gkdv_st}, we see that $S''(\psi)(v_{0} + \mu v_{2}) = 0$, which implies that there exists $w \in \operatorname{ker} S''(\psi)$ which satisfies $v_{0} = - \mu v_{2} + w$.
    By \ref{cond:C5;gkdv_st}, we obtain
    \begin{equation}
        0 = \langle f, v_{0} \rangle = \langle f, -\mu v_{2} + w \rangle = -\mu \langle f, v_{2} \rangle.
    \end{equation}
    Combining this with $\langle f, v_{2} \rangle > 0$, we see that $\mu = 0$, which gives $v_{0} = w$.
    This means that $0 \neq v_{0} \in W \cap \operatorname{ker} S''(\psi) = \{ 0 \}$, which is impossible.

    Hence, the proof is completed.
\end{proof}

Finally, we prove Lemma \ref{lem:coercivity_single;gkdv_st} by applying Lemma \ref{lem:suffcond2_abstract_statement;gkdv_st}.
\begin{proof}
    First, we show that the functionals $S_{1}^{0, r}$ and $K_{1}^{0, r}$ satisfy the condition \ref{cond:C1;gkdv_st}--\ref{cond:C5;gkdv_st} with $X = H^{\sigma/2}(\mathbb{R})$ and $f = M'(\psi_{1, r}) = \psi_{1, r}$.

    \ref{cond:C1;gkdv_st}:
    We can take $v_{1} = \psi_{1, r}$.
    Indeed, since $\psi_{1, r}$ solves \eqref{eq:coercivity0500;gkdv_st}, we immediately see that $\langle (S_{1}^{0, r})'(\psi_{1, r}), \psi_{1, r} \rangle = K_{1}^{0, r}(\psi_{1, r}) = 0$.
    Moreover, considering the graph of the function $(0, \infty) \ni \lambda \mapsto K_{1}^{0, r}(\lambda \psi_{1, r})$, we see that
    \begin{equation}
        \langle (K_{1}^{0, r})'(\psi_{1, r}), \psi_{1, r} \rangle = \left. \partial_{\lambda} K_{1}^{0, r}(\lambda \psi_{1, r}) \right|_{\lambda = 1} < 0.
    \end{equation}

    \ref{cond:C2;gkdv_st}: We introduce the following stationary problem:
    \begin{equation}
        D_{x}^{\sigma} \psi + c \psi - \psi^{r} = 0, \quad x \in \mathbb{R}, \label{eq:suffcond2_3000;gkdv_st}
    \end{equation}
    where $c > 0$.
    Here we put
    \begin{align}
        & S_{c}^{0, r}(v) = \frac{1}{2} \| v \|_{H^{\sigma/2}_{c}}^{2} - \frac{1}{r+1} \int_{\mathbb{R}} v^{r+1} \, dx, \\
        & K_{c}^{0, r}(v) = \langle (S_{c}^{0, r})'(v), v \rangle,
    \end{align}
    which are the action functional and the Nehari functional corresponding to \eqref{eq:suffcond2_3000;gkdv_st}, respectively.
    Let $\psi_{c, r}$ be the positive ground state solution to \eqref{eq:suffcond2_3000;gkdv_st}.
    Then, we can see that
    \begin{equation}
        \psi_{c, r}(x) = c^{1/(r-1)}\psi_{1, r}(c^{1/\sigma} x). \label{eq:suffcond2_3100;gkdv_st}
    \end{equation}
    Differentiating this with respect to $c$, we obtain
    \begin{equation}
        \left. \partial_{c} \psi_{c, r} \right|_{c = 1} = \frac{1}{r-1} \psi_{1, r} + \frac{1}{\sigma} x \partial_{x}\psi_{1, r}.
    \end{equation}
    Then $\left. \partial_{c} \psi_{c, r} \right|_{c = 1} \in H^{\sigma/2}(\mathbb{R})$ follows from Lemma \ref{lem:suffcond2_regularity_energysp;gkdv_st}.

    Moreover, since $(S_{c}^{0, r})'(\psi_{c}^{0, r}) = 0$ for all $c > 0$, we have
    \begin{align}
        0 = \partial_{c} (S_{c}^{0, r})'(\psi_{c, r}) = (S_{c}^{0, r})''(\psi_{c, r}) \partial_{c} \psi_{c, r} + \psi_{c, r}.
    \end{align}
    Then we obtain $\left. (S_{1}^{0, r})''(\psi_{1, r})\partial_{c} \psi_{c, r} \right|_{c = 1} = - \psi_{1, r}$.

    Additionally, using \eqref{eq:suffcond2_3100;gkdv_st}, we have
    \begin{equation}
        M(\psi_{c, r}) = c^{2/(r-1) - 1/\sigma}M(\psi_{1, r}).
    \end{equation}
    Then we obtain
    \begin{equation}
        \langle M'(\psi_{1, r}), \left. \partial_{c} \psi_{c, r} \right|_{c = 1} \rangle = \left. \partial_{c} M(\psi_{c, r}) \right|_{c=1} = \left( \frac{2}{r-1} - \frac{1}{\sigma} \right) M(\psi_{1, r}).
    \end{equation}
    Therefore, if $r < 2 \sigma + 1$, then $\langle M'(\psi_{1, r}), \left. \partial_{c}\psi_{c, r} \right|_{c = 1} \rangle > 0$.

    Finally, since $K_{c}^{0, r}(\psi_{c, r}) = 0$ holds for all $c > 0$, we have
    \begin{equation}
        0 = \left. \partial_{c} K_{c}^{0, r}(\psi_{c, r}) \right|_{c = 1} = \langle (K_{1}^{0, r})'(\psi_{1, r}), \left. \partial_{c} \psi_{c, r} \right|_{c = 1} \rangle + \| \psi_{1, r} \|_{L^{2}}^{2},
    \end{equation}
    which gives that $\langle (K_{1}^{0, r})'(\psi_{1, r}), \left. \partial_{c} \psi_{c, r} \right|_{c = 1} \rangle \neq 0$.

    Hence, we take $v_{2} = \left. \partial_{c} \psi_{c, r} \right|_{c=1}$ to conclude the proof of \ref{cond:C2;gkdv_st}.

    \ref{cond:C3;gkdv_st}: Let $(v_{n})_{n} \in H^{\sigma/2}(\mathbb{R})$ satisfy $v_{n} \rightharpoonup v$ weakly in $H^{\sigma/2}(\mathbb{R})$ with some $v \in H^{\sigma/2}(\mathbb{R})$.
    Since $\psi_{1, r}(x) \rightarrow 0$ as $|x| \rightarrow +\infty$,
    we can see that
    \begin{equation}
        \int_{\mathbb{R}} \psi_{1, r}^{r-1} v_{n}^{2} \, dx \rightarrow \int_{\mathbb{R}} \psi_{1, r}^{r-1} v^{2} \, dx. \label{eq:suffcond2_3200;gkdv_st}
    \end{equation}
    This convergence and the Fatou lemma yield that
    \begin{equation}
        \langle (S_{1}^{0, r})''(\psi_{1, r})v, v \rangle \leq \liminf_{n \rightarrow \infty} \langle (S_{1}^{0, r})''(\psi_{1, r})v_{n}, v_{n} \rangle,
    \end{equation}
    which implies the weak lower semicontinuity.

    \ref{cond:C4;gkdv_st}: We can easily see that \ref{cond:C4;gkdv_st} holds from \eqref{eq:suffcond2_3200;gkdv_st}.

    \ref{cond:C5;gkdv_st}: First, it is known that $\operatorname{ker}(S_{1}^{0, r})''(\psi_{1, r}) = \operatorname{span} \{ \partial_{x} \psi_{1, r} \}$ holds.
    For details, see Frank--Lenzman~\cite{Frank-Lenzmann}.

    Since $M(\psi_{1, r}(\cdot + y)) = M(\psi_{1, r})$ holds for all $y \in \mathbb{R}$, we have
    \begin{equation}
        0 = \left. \partial_{y} M(\psi_{1, r}(\cdot + y)) \right|_{y = 0} = \langle M'(\psi_{1, r}), \partial_{x} \psi_{1, r} \rangle = (\psi_{1, r}, \partial_{x} \psi_{1, r})_{L^{2}}.
    \end{equation}
    Then $\langle M'(\psi_{1, r}), v \rangle = 0$ holds for all $v \in \operatorname{ker}(S_{1}^{0,r})''(\psi_{1, r})$.

    Next we introduce a suitable subspace $W$ of $H^{\sigma/2}(\mathbb{R})$.
    We put
    \begin{align}
        W &\coloneq \{ v \in H^{\sigma/2}(\mathbb{R}): (v, \partial_{x} \psi_{1, r})_{L^{2}} = 0 \} \\
        &= \{ v \in H^{\sigma/2}(\mathbb{R}): (v, w)_{L^{2}} = 0 \ \text{for all} \ w \in \operatorname{ker}(S_{1, r})''(\psi_{1, r}) \}.
    \end{align}
    Then we can see that $W$ is a subspace of $H^{\sigma/2}(\mathbb{R})$ and satisfies $W \cap \operatorname{ker}(S_{1}^{0, r})''(\psi_{1, r}) = \{0\}$.
    Therefore, by Lemma \ref{lem:suffcond2_abstract_statement;gkdv_st}, there exists $C_{2} > 0$ such that
    \begin{equation}
        \langle (S_{1}^{0, r})''(\psi_{1, r}) v, v \rangle \geq C_{2} \| v \|_{H^{\sigma/2}}^{2} \label{eq:suffcond2_3300;gkdv_st}
    \end{equation}
    holds for all $v \in W$ satisfying $\langle M'(\psi_{1, r}), v \rangle = 0$.
    Namely, \eqref{eq:suffcond2_3300;gkdv_st} holds for all $v \in H^{\sigma/2}(\mathbb{R})$ satisfying $(v, \psi_{1, r})_{L^{2}} = (v, \partial_{x} \psi_{1, r})_{L^{2}} = 0$.

    Hence, the proof is accomplished.
\end{proof}

\subsection{Proof of Lemma \ref{lem:convergence;gkdv_st}}
In this subsection, we prove the convergence properties of ground state solutions to \eqref{eq:Intro0300;gkdv_st}.

First, we consider case (I), where $a = +1$ and $q$ is odd.
Here we define the Nehari functional $\tilde{K}_{c}$ derived from $\tilde{S}_{c}$ as
\begin{equation}
    \tilde{K}_{c}(v) \coloneq \langle \tilde{S}_{c}'(v), v \rangle = \| v \|_{H^{\sigma/2}}^{2} - \int_{\mathbb{R}} v^{p+1} \, dx - c^{\alpha} \| v \|_{L^{q+1}}^{q+1}
\end{equation}
for $v \in H^{\sigma/2}(\mathbb{R})$.
Similarly to \eqref{eq:GS0230;gkdv_st}, we put
\begin{align}
    \tilde{J}_{c}(v) & \coloneq \tilde{S}_{c}(v) - \frac{1}{p+1} \tilde{K}_{c}(v) \\
    &= \left( \frac{1}{2} - \frac{1}{p+1} \right) \| v \|_{H^{\sigma/2}}^{2} + c^{\alpha} \left( \frac{1}{p+1} - \frac{1}{q+1} \right) \| v \|_{L^{q+1}}^{q+1}.
\end{align}
Then we can see that the following characterizations are equivalent:
\begin{align}
    \tilde{S}_{c}(\tilde{\phi}_{c}) &= \inf \{\tilde{S}_{c}(v): v \in H^{\sigma/2}(\mathbb{R}) \setminus \{0\}, \ \tilde{K}_{c}(v) = 0 \}, \\
    &= \inf \{ \tilde{J}_{c}(v): v \in H^{\sigma/2}(\mathbb{R}) \setminus \{0\}, \ \tilde{K}_{c}(v) = 0 \}, \\
    &= \inf \{ \tilde{J}_{c}(v): v \in H^{\sigma/2}(\mathbb{R}) \setminus \{0\}, \ \tilde{K}_{c}(v) \leq 0 \}, \label{eq:suffcond2_2100;gkdv_st}
\end{align}
where $\phi_{c}$ is a positive ground state solution to \eqref{eq:Intro0300;gkdv_st} and $\tilde{\phi_{c}}$ is a function obtained by the scaling \eqref{eq:suffcond2_0100;gkdv_st}.
Similar characterizations of the ground state solution hold for the stationary problem \eqref{eq:coercivity0500;gkdv_st}.
Namely, we put
\begin{equation}
    J_{1}^{0, r}(v) \coloneq S_{1}^{0, r}(v) - \frac{1}{r+1} K_{1}^{0, r}(v) = \left( \frac{1}{2} - \frac{1}{r+1} \right) \| v \|_{H^{\sigma/2}}^{2}.
\end{equation}
Then, we can see that
\begin{align}
    S_{1}^{0,r}(\psi_{1, r}) &= \inf \{ S_{1}^{0,r}(v): v \in H^{\sigma/2}(\mathbb{R}), \ K_{1}^{0,r}(v) = 0 \}, \\
    &= \inf \{ J_{1}^{0,r}(v): v \in H^{\sigma/2}(\mathbb{R}), \ K_{1}^{0,r}(v) = 0 \}, \\
    &= \inf \{ J_{1}^{0,r}(v): v \in H^{\sigma/2}(\mathbb{R}), \ K_{1}^{0,r}(v) \leq 0 \}, \label{eq:suffcond2_2200;gkdv_st}
\end{align}
where $\psi_{1, r}$ is the positive ground state solution to \eqref{eq:suffcond2_0600;gkdv_st} and $r \in \{p, q\}$.
The first and second equalities in \eqref{eq:suffcond2_2100;gkdv_st} and \eqref{eq:suffcond2_2200;gkdv_st} immediately holds.
The third equalities in \eqref{eq:suffcond2_2100;gkdv_st} and \eqref{eq:suffcond2_2200;gkdv_st} can be shown via a similar discussion in Lemma \ref{lem:GS_ineq_IJ;gkdv_st}.

\begin{lem} \label{lem:boundedness_of_seq:gkdv_st}
    Assume condition {\rm (I)} in Theorem \ref{thm:existence_of_GS;gkdv_st}.
    Then there exists $M_{1} > 0$ such that $\| \tilde{\phi}_{c} \|_{H^{\sigma/2}} \leq M_{1}$ for all $c > 0$.
\end{lem}
\begin{proof}
    By $K_{1}^{0,p}(\psi_{1,p}) = 0$, we see that
    \begin{align}
        \tilde{K}_{c}(\psi_{1, p}) &= \| \psi_{1, p} \|_{H^{\sigma/2}}^{2} - \int_{\mathbb{R}} \psi_{1, p}^{p+1} \, dx - c^{\alpha} \| \psi_{1, p} \|_{L^{q+1}}^{q+1} = - c^{\alpha} \| \psi_{1, p} \|_{L^{q+1}}^{q+1} < 0
    \end{align}
    holds for all $c > 0$.
    Combining this with \eqref{eq:suffcond2_2100;gkdv_st}, we obtain
    \begin{equation}
        \tilde{J}_{c}(\tilde{\phi}_{c}) \leq \tilde{J}_{c}(\psi_{1,p}) = \left( \frac{1}{2} - \frac{1}{p+1} \right) \| \psi_{1, p} \|_{H^{\sigma/2}}^{2} + c^{\alpha} \left( \frac{1}{p+1} - \frac{1}{q+1} \right) \| \psi_{1, p} \|_{L^{q+1}}^{q+1}. \label{eq:suffcond2_2210;gkdv_st}
    \end{equation}
    Since the right hand side of \eqref{eq:suffcond2_2210;gkdv_st} converges to $J_{1}^{0, p}(\psi_{1, p})$ as $c \rightarrow +0$, there exists $M_{0} > 0$ such that $\tilde{J}_{c}(\tilde{\phi}_{c}) \leq M_{0}$ holds for all $c > 0$.
    Finally, by the definition of $\tilde{J}_{c}$, we find some constant $C > 0$ independent of $c$ such that
    \begin{equation}
        \| \tilde{\phi}_{c} \|_{H^{\sigma/2}}^{2} \leq C \tilde{J}_{c}(\tilde{\phi}_{c}) \leq C M_{0}.
    \end{equation}
    Then, replacing $(CM_{0})^{1/2}$ with $M_{1}$ concludes the proof.
\end{proof}

\begin{lem} \label{lem:suffcond2_Nehari_est;gkdv_st}
    Assume condition {\rm (I)} in Theorem \ref{thm:existence_of_GS;gkdv_st}.
    Then, for any $\mu > 1$, there exists $\tilde{c} \in (0, \infty)$ such that $K_{1}^{0, p}(\mu \tilde{\phi}_{c}) < 0$ holds for all $c \in (0, \tilde{c})$.
\end{lem}

\begin{proof}
    Let $\mu > 1$ be arbitrary.
    By $\tilde{K}_{c}(\tilde{\phi}_{c}) = 0$, the Sobolev embedding, and Lemma \ref{lem:boundedness_of_seq:gkdv_st}, we see that
    \begin{align}
        \mu^{-2} K_{1}^{0,p}(\mu \tilde{\phi}_{c}) &= \| \tilde{\phi}_{c} \|_{H^{\sigma/2}}^{2} - \mu^{p-1} \int_{\mathbb{R}} \tilde{\phi}_{c}^{p+1} \, dx \\
        &= -(\mu^{p-1} - 1) \| \tilde{\phi}_{c} \|_{H^{\sigma/2}}^{2} + \mu^{p-1} c^{\alpha} \| \tilde{\phi}_{c} \|_{L^{q+1}}^{q+1} \\\
        &\leq - \| \tilde{\phi}_{c} \|_{H^{\sigma/2}}^{2} \left\{ (\mu^{p-1} - 1) - C \mu^{p-1} c^{\alpha} \| \tilde{\phi}_{c} \|_{H^{\sigma/2}}^{q-1} \right\} \\
        &\leq - \| \tilde{\phi}_{c} \|_{H^{\sigma/2}}^{2} \left\{ (\mu^{p-1} - 1) - C M_{1}^{p-1} \mu^{p-1} c^{\alpha} \right\} \label{eq:suffcond2_2300;gkdv_st}
    \end{align}
    with some constant $C > 0$.
    Then, taking $\tilde{c} > 0$ so small that $K_{1}^{0, p}(\mu \tilde{\phi}_{c}) < 0$ holds for $c \in (0, \tilde{c})$.
\end{proof}

\begin{lem} \label{lem:suffcond2_min_seq;gkdv_st}
    Assume condition {\rm (I)} in Theorem \ref{thm:existence_of_GS;gkdv_st}.
    Then the followings hold:

    \begin{enumerate}[label={\rm (\roman*)} \ ]
        \item $\displaystyle \lim_{c \rightarrow +0} J_{1}^{0, p}(\tilde{\phi}_{c}) = J_{1}^{0, p}(\psi_{1, p})$;
        \item $\displaystyle \lim_{c \rightarrow + 0} K_{1}^{0, p}(\tilde{\phi}_{c}) = 0$.
    \end{enumerate}
\end{lem}
\begin{proof}
    (i) \
    First, \eqref{eq:suffcond2_2210;gkdv_st} implies $\limsup_{c \rightarrow +0} \tilde{J}_{c}(\tilde{\phi}_{c}) \leq J_{1}^{0, p}(\psi_{1, p})$.
    
    Let $\mu > 1$.
    By \eqref{eq:suffcond2_2200;gkdv_st} and Lemma \ref{lem:suffcond2_Nehari_est;gkdv_st}, we have
    \begin{equation}
        J_{1}^{0, p}(\psi_{1, p}) \leq J_{1}^{0, p}(\mu \tilde{\phi}_{c}) \leq \mu^{2} \tilde{J}_{c}(\tilde{\phi}_{c})
    \end{equation}
    for $c > 0$ sufficiently small.
    By considering the arbitrary of $\mu > 1$, we obtain
    \begin{equation}
        J_{1}^{0, p}(\psi_{1, p}) \leq \liminf_{c \rightarrow +0} \tilde{J}_{c}(\tilde{\phi}_{c}).
    \end{equation}
    Thus, we conclude that $\lim_{c \rightarrow +0} \tilde{J}_{c}(\tilde{\phi}_{c}) = J_{1}^{0, p}(\psi_{1, p})$.

    Since we can write
    \begin{equation}
        J_{1}^{0, p}(v) = \tilde{J}_{c}(v) - c^{\alpha} \left( \frac{1}{p+1} - \frac{1}{q+1} \right) \| v \|_{L^{q+1}}^{q+1},
    \end{equation}
    we obtain
    \begin{equation}
        \lim_{c \rightarrow +0} J_{1}^{0, p}(\tilde{\phi}_{c}) = \lim_{c \rightarrow +0} \tilde{J}_{c}(\tilde{\phi}_{c}) = J_{1}^{0, p}(\psi_{1, p}).
    \end{equation}

    (ii) \ By $\tilde{K}_{c}(\tilde{\phi}_{c}) = 0$, we have
    \begin{equation}
        K_{1}^{0, p}(\tilde{\phi}_{c}) = c^{\alpha} \| \tilde{\phi}_{c} \|_{L^{q+1}}^{q+1}.
    \end{equation}
    Since the family $\{ \tilde{\phi}_{c} \}_{c > 0}$ is also bounded in $L^{q+1}(\mathbb{R})$, we can conclude that $\lim_{c \rightarrow +0} K_{1}^{0}(\tilde{\phi}_{c}) = 0$.
\end{proof}

To prove Lemma \ref{lem:convergence;gkdv_st}, we recall a compactness lemma obtained by Strauss~\cite{Strauss}.

\begin{lem}[Strauss~\cite{Strauss}] \label{lem:Strauss_Compactness;gkdv_st}
    Assume that $P, \ Q\colon \mathbb{R} \rightarrow \mathbb{R}$ are continuous functions satisfying
    \begin{align}
        \frac{P(s)}{Q(s)} \rightarrow 0 \quad \text{\rm as} \ |s| \rightarrow \infty \ \text{\rm and} \ |s| \rightarrow 0.
    \end{align}
    Moreover, let $(u_{n})_{n}$ be a sequence of measurable functions defined in $\mathbb{R}$ which satisfies
    \begin{align}
        \sup_{n\in\mathbb{N}}\int_{\mathbb{R}}|Q(u_{n})| \, dx < \infty
    \end{align}
    and assume that $P(u_{n}(x))\rightarrow v(x)$ a.e.\ in $\mathbb{R}$ with some measurable function $v \colon \mathbb{R} \rightarrow \mathbb{R}$, and that $u_{n}(x)\rightarrow 0$ as $|x|\rightarrow \infty$ uniformly with respect to $n\in\mathbb{N}$.
    Then it holds that $P(u_{n})\rightarrow v$ in $L^{1}(\mathbb{R})$.
\end{lem}

\begin{lem} \label{lem:uniform_decay_of_even_functions;gkdv_st}
    Let $(u_{n})_{n} \subset H^{\sigma/2}(\mathbb{R})$ be a sequence of nonnegative and even function which decrease in $|x|$.
    If $(u_{n})_{n}$ is bounded in $H^{\sigma/2}(\mathbb{R})$, then there exists $C > 0$ which is independent of $n \in \mathbb{R}$ and satisfies
    \begin{equation}
        \sup_{n \in \mathbb{N}}u_{n}(x) \leq C|x|^{-1/2}
    \end{equation}
    for all $x \in \mathbb{R}$.
\end{lem}
\begin{proof}
    Here we put $M \coloneqq \sup_{n\in \mathbb{N}} \| u_{n} \|_{H^{\sigma/2}}$.
    Without loss of generality, we may assume that $x > 0$.
    Then, for all $n \in \mathbb{N}$, direct calculation yields that
    \begin{equation}
        \| u_{n} \|_{L^{2}}^{2} = 2 \int_{0}^{\infty} u_{n}(y)^{2} \, dy \geq 2 \int_{0}^{x} u_{n}(y)^{2} \, dy \geq 2xu_{n}(x)^2.
    \end{equation}
    Therefore, we obtain
    \begin{equation}
        u_{n}(x) \leq 2^{-1/2} x^{-1/2} \| u_{n} \|_{L^{2}}^{2} \leq C x^{-1/2} \| u_{n} \|_{H^{\sigma/2}}^{2} \leq C M^{2} x^{-1/2}
    \end{equation}
    holds for all $n \in \mathbb{N}$.
    Replacing $CM^{2}$ with $C$ concludes the proof.
\end{proof}

Finally, we give the proof of Lemma \ref{lem:convergence;gkdv_st}.
\begin{proof}
    First, we remark that the minimizers of \eqref{eq:suffcond2_2200;gkdv_st} coincide with the ground state solutions to \eqref{eq:coercivity0500;gkdv_st}.
    We can see it with similar way to Section \ref{subsection:GS_existence;gkdv_st}.

    Let $(c_{n})_{n} \subset (0, \infty)$ satisfy $c_{n} \rightarrow +0$.
    By Lemma \ref{lem:suffcond2_min_seq;gkdv_st}, we have
    \begin{align}
        & J_{1}^{0, p}(\tilde{\phi}_{c_{n}}) \rightarrow J_{1}^{0, p}(\psi_{1, p}), \label{eq:suffcond2_0280;gkdv_st} \\
        & K_{1}^{0, p}(\tilde{\phi}_{c_{n}}) \rightarrow 0. \label{eq:suffcond2_0290;gkdv_st}
    \end{align}
    Since \eqref{eq:suffcond2_0280;gkdv_st} implies that $(\tilde{c}_{n})_{n}$ is bounded in $H^{\sigma/2}(\mathbb{R})$, up to a subsequence, there exists $v_{0} \in H^{\sigma/2}(\mathbb{R})$ such that $\tilde{\phi}_{c_{n}} \rightharpoonup v_{0}$ weakly in $H^{\sigma/2}(\mathbb{R})$.
    This yields that
    \begin{equation}
        J_{1}^{0, p}(v_{0}) \leq \liminf_{n \rightarrow \infty}J_{1}^{0, p}(\tilde{\phi}_{c_{n}}) = J_{1}^{0, p}(\psi_{1, p}). \label{eq:suffcond2_0291;gkdv_st}
    \end{equation}

    Next, we show that $\tilde{\phi}_{c_{n}} \rightarrow v_{0}$ strongly in $L^{\gamma+1}$ for $\gamma \in (1, \infty)$.
    By the weak convergence, we can see that $\tilde{\phi}_{c_{n}}(x) \rightarrow v_{0}(x)$ a.e.\ in $\mathbb{R}$.
    Moreover, we put $P(s) \coloneq |s|^{r+1}$, $Q(s) \coloneq s^{2} + |s|^{r+2}$ for $s \in \mathbb{R}$, and $w_{n} \coloneq \tilde{\phi}_{c_{n}} - v_{0}$ for $n \in \mathbb{N}$.
    Applying Lemma \ref{lem:uniform_decay_of_even_functions;gkdv_st}, we can verify that the functions $P$ and $Q$, and the sequence $(w_{n})_{n}$ satisfy the assumptions of Lemma \ref{lem:Strauss_Compactness;gkdv_st}.
    Therefore, we obtain $\| w_{n} \|_{L^{\gamma+1}} \rightarrow 0$, which implies that $\tilde{\phi}_{c_{n}} \rightarrow v_{0}$ strongly in $L^{\gamma+1}(\mathbb{R})$.
    Using this strong convergence and \eqref{eq:suffcond2_0280;gkdv_st}, we obtain
    \begin{equation}
        K_{1}^{0, p}(v_{0}) \leq \liminf_{n \rightarrow \infty} K_{1}^{0, p}(\tilde{\phi}_{c_{n}}) = 0. \label{eq:suffcond2_2310;gkdv_st}
    \end{equation}
    Suppose that $K_{1}^{0, p}(v_{0}) < 0$.
    Then, similarly to Lemma \ref{lem:GS_ineq_IJ;gkdv_st}, we can find some $\lambda_{0} \in (0, 1)$ such that $K_{1}^{0, p}(\lambda_{0} v_{0}) = 0$.
    By \eqref{eq:suffcond2_2200;gkdv_st}, we obtain
    \begin{equation}
        J_{1}^{0, p}(\psi_{1, p}) \leq J_{1}^{0}(\lambda_{0} v_{0}) < J_{1}^{0, p}(v_{0}),
    \end{equation}
    which contradicts to \eqref{eq:suffcond2_0291;gkdv_st}.
    Therefore, we conclude that $K_{1}^{0, p}(v_{0}) = 0$.
    Transforming this, we have $\| v_{0} \|_{H^{\sigma/2}}^{2} = \int_{\mathbb{R}} v_{0}^{p+1} \, dx$.
    Moreover, we see it from $\tilde{K}_{c}(\tilde{\phi}_{c_{n}}) = 0$ that
    \begin{equation}
        \| \tilde{\phi}_{c_{n}} \|_{H^{\sigma/2}}^{2} = \int_{\mathbb{R}} \tilde{\phi}_{c_{n}}^{p+1} \, dx + c^{\alpha} \| \tilde{\phi}_{c_{n}} \|_{L^{q+1}}^{q+1} \rightarrow \int_{\mathbb{R}} v_{0}^{p+1} \, dx \quad \text{as} \ n \rightarrow \infty,
    \end{equation}
    which yields that $\| \tilde{\phi}_{c_{n}} \|_{H^{\sigma/2}} \rightarrow \| v_{0} \|_{H^{\sigma/2}}$.
    Since $\tilde{\phi}_{c_{n}} \rightharpoonup v_{0}$ weakly in $H^{\sigma/2}(\mathbb{R})$, we conclude that $\tilde{\phi}_{c_{n}} \rightarrow v_{0}$ strongly in $H^{\sigma/2}(\mathbb{R})$.
    By \eqref{eq:suffcond2_0280;gkdv_st}, we obtain
    \begin{equation}
        J_{1}^{0, p}(v_{0}) = \lim_{n \rightarrow \infty}J_{1}^{0, p}(\tilde{\phi}_{c_{n}}) = J_{1}^{0, p}(\psi_{1, p}) > 0,
    \end{equation}
    which implies that $v_{0} \neq 0$ and that $S_{1}^{0, p}(v_{0}) = S_{1}^{0, p}(\psi_{1, p})$.
    Since $K_{1}^{0, p}(v_{0}) = 0$, we see that $v_{0}$ is a minimizer of $S_{1}^{0, p}(\psi_{1, p})$, which means that $v_{0}$ is a ground state solution to \eqref{eq:coercivity0500;gkdv_st}.
    Moreover, since $\tilde{\phi}_{c_{n}}$ is positive and even for each $n \in \mathbb{N}$ and that $\tilde{\phi}_{c_{n}}(x) \rightarrow v_{0}(x)$ a.e.\ in $\mathbb{R}$, we see that $v_{0}$ is also positive and even.
    Therefore, by the uniqueness of the positive and even ground state solution to \eqref{eq:coercivity0500;gkdv_st}, we obtain $v_{0} = \psi_{1, p}$.

    Thus, we conclude that $\tilde{\phi}_{c_{n}} \rightarrow \psi_{1, p}$ strongly in $H^{\sigma/2}(\mathbb{R})$.
\end{proof}

Next, we consider case (II-1) and $p$ is odd.
The Nehari functional $\breve{K}_{c}$ derived from $\breve{S}_{c}$ is given as
\begin{equation}
    \breve{K}_{c}(v) \coloneq \langle \breve{S}_{c}'(v), v \rangle = \| v \|_{H^{\sigma/2}}^{2} - c^{-\beta} \| v \|_{L^{p+1}}^{p+1} - \int_{\mathbb{R}} v^{q+1} \, dx.
\end{equation}
Here we define similar functional $\breve{I}_{c}$ to \eqref{eq:GS0220;gkdv_st} as
\begin{align}
    \breve{I}_{c}(v) &\coloneq \breve{S}_{c}(v) - \frac{1}{q+1} \breve{K}_{c}(v) \\
    &= \left( \frac{1}{2} - \frac{1}{q+1} \right) \| v \|_{H^{\sigma/2}}^{2} + c^{- \beta} \left( \frac{1}{p+1} - \frac{1}{q+1} \right) \| v \|_{L^{p+1}}^{p+1}.
\end{align}
Then, similarly to \eqref{eq:suffcond2_2100;gkdv_st}, we can see that
\begin{align}
    \breve{S}_{c}(\breve{\phi}_{c}) &= \inf \{ \breve{S}_{c}(v): v \in H^{\sigma/2}(\mathbb{R}) \setminus \{0\}, \ \breve{K}_{c}(v) = 0 \}, \\
    &= \inf \{ \breve{I}_{c}(v): v \in H^{\sigma/2}(\mathbb{R}) \setminus \{0\}, \ \breve{K}_{c}(v) = 0 \}, \\
    &= \inf \{ \breve{I}_{c}(v): v \in H^{\sigma/2}(\mathbb{R}) \setminus \{0\}, \ \breve{K}_{c}(v) \leq 0 \}, \label{eq:suffcond2_2400;gkdv_st}
\end{align}
where $\phi_{c}$ is a positive ground state solution to \eqref{eq:Intro0300;gkdv_st}, and $\breve{\phi}_{c}$ is a function obtained by the scaling \eqref{eq:coercivity0300;gkdv_st}.

To prove Lemma \ref{lem:convergence;gkdv_st} in this case, it is sufficient to
show the following statement:
\begin{lem} \label{label:min_seq_caseII1;gkdv_st}
    Assume condition {\rm (II-1)} in Theorem \ref{thm:existence_of_GS;gkdv_st}.
    Let $\phi_{c}$ be a positive ground state solution to \eqref{eq:Intro0300;gkdv_st} $c > 0$, and $\breve{\phi}_{c}$ is a function obtained by the scaling \eqref{eq:coercivity0300;gkdv_st}.
    Moreover, let $\psi_{1,q}$ be the positive ground state solution to \eqref{eq:coercivity0500;gkdv_st} with $r = q$.
    Then, the followings hold:
    \begin{enumerate}[label={\rm (\roman*)} \ ]
        \item $\displaystyle \lim_{c \rightarrow +\infty} J_{1}^{0, q}(\breve{\phi}_{c}) = J_{1}^{0, q}(\psi_{1, q})$;
        \item $\displaystyle \lim_{c \rightarrow +\infty} K_{1}^{0,q}(\breve{\phi}_{c}) = 0$.
    \end{enumerate}
\end{lem}
The method to prove this lemma is almost the same as that of Lemma \ref{lem:suffcond2_min_seq;gkdv_st}, but the way to obtain the boundedness of the family $\{\breve{\phi}_{c}\}_{c > 0}$ in $H^{\sigma/2}$ is slightly different.
\begin{proof}[Proof of Lemma \ref{label:min_seq_caseII1;gkdv_st}]
    Let $\mu > 1$ be arbitrary.
    By $K_{1}^{0,q}(\psi_{1,q}) = 0$, we have
    \begin{align}
        \mu^{-2} \breve{K}_{c}(\mu \psi_{1, q}) &= \| \psi_{1, q} \|_{H^{\sigma/2}}^{2} + \mu^{p-1} c^{-\beta} \| \psi_{1, q} \|_{L^{p+1}}^{p+1} - \mu^{q-1} \int_{\mathbb{R}} \psi_{1, q}^{q+1} \, dx \\
        &\leq - \| \psi_{1, q} \|_{H^{\sigma/2}}^{2} \left\{ (\mu^{q-1} - 1) - C \mu^{q-1} c^{-\beta} \| \psi_{1, q} \|_{H^{\sigma/2}}^{p-1} \right\}
    \end{align}
    with some constant $C > 0$.
    Then we can see that there exists $\breve{c}_{1} \in (0, \infty)$ such that $\breve{K}_{c}(\mu \psi_{1, q}) < 0$ for $c \in (\breve{c}_{1}, \infty)$.
    Therefore, by \eqref{eq:suffcond2_2400;gkdv_st}, we obtain
    \begin{equation}
        \breve{I}_{c}(\breve{\phi}_{c}) \leq \breve{I}_{c}(\mu \psi_{1, q}) \leq \mu^{q+1} \left\{ \left( \frac{1}{2} - \frac{1}{q+1} \right) \| \psi_{1, q} \|_{H^{\sigma/2}}^{2} + c^{-\beta} \left( \frac{1}{p+1} - \frac{1}{q+1} \right) \| \psi_{1,q} \|_{H^{\sigma/2}}^{p+1} \right\}. \label{eq:suffcond2_2500;gkdv_st}
    \end{equation}
    The right hand side of \eqref{label:min_seq_caseII1;gkdv_st} converges to $\mu^{q+1} J_{1}^{0,q}(\psi_{1, q})$ as $c \rightarrow +\infty$.
    Then the arbitrarity of $\mu > 1$ gives
    \begin{equation}
        \limsup_{c \rightarrow +\infty} \breve{I}_{c}(\breve{\phi}_{c}) \leq J_{1}^{0, q}(\psi_{1,q}). \label{eq:suffcond2_2600;gkdv_st}
    \end{equation}
    Moreover, by \eqref{eq:suffcond2_2500;gkdv_st}, we can see that
    \begin{equation}
        \| \breve{\phi}_{c} \|_{H^{\sigma/2}}^{2} \leq C \mu^{q+1} \left( \| \psi_{1, q} \|_{H^{\sigma/2}}^{2} + c^{-\beta} \| \psi_{1, q} \|_{H^{\sigma/2}}^{p+1} \right)
    \end{equation}
    holds for $c > \breve{c}_{1}$.

    Next, since $\breve{K}_{c}(\breve{\phi}_{c}) = 0$, we have
    \begin{align}
        \mu^{-2} K_{1}^{0,q}(\breve{\phi}_{c}) &= \| \breve{\phi}_{c} \|_{H^{\sigma/2}}^{2} - \mu^{q-1} \int_{\mathbb{R}} \breve{\phi}_{c}^{q+1} \, dx \\
        &= - (\mu^{q-1} - 1) \| \breve{\phi}_{c} \|_{H^{\sigma/2}}^{2} + \mu^{q-1} c^{-\beta} \| \breve{\phi}_{c} \|_{L^{p+1}}^{p+1} \\
        &\leq - \| \breve{\phi}_{c} \|_{H^{\sigma/2}}^{2} \left\{ (\mu^{q-1} - 1) - C \mu^{q-1} c^{-\beta} \| \breve{\phi}_{c} \|_{H^{\sigma/2}}^{p-1} \right\} \\
        &\leq - \| \breve{\phi}_{c} \|_{H^{\sigma/2}}^{2} \left\{(\mu^{q-1} - 1) - C \mu^{(q-1)(q+1)\theta} \left( c^{-\beta} \| \psi_{1,q} \|_{H^{\sigma/2}}^{p-1} + c^{-\beta(\theta + 1)} \| \psi_{1,q} \|_{H^{\sigma/2}}^{(p+1)\theta} \right) \right\},
    \end{align}
    where $\theta = (p-1)/2$.
    Then there exists $\breve{c}_{2} \in [\breve{c}_{1}, \infty)$ such that $K_{1}^{0, q}(\mu \breve{\phi}_{c}) < 0$ holds for $c > \breve{c}_{2}$.
    Therefore, by \eqref{eq:suffcond2_2100;gkdv_st}, we obtain
    \begin{equation}
        J_{1}^{0,q}(\psi_{1,q}) \leq J_{1}^{0,q}(\mu \breve{\phi}_{c}) \leq \mu^{2} \breve{I}_{c}(\breve{\phi})
    \end{equation}
    for $c > \breve{c}_{2}$.
    Finally, since $\mu > 1$ is arbitrary, we see that
    \begin{equation}
        J_{1}^{0,q}(\psi_{1,q}) \leq \liminf_{c \rightarrow +\infty} \breve{I}_{c}(\breve{\phi}_{c}).
    \end{equation}
    Combining this with \eqref{eq:suffcond2_2600;gkdv_st}, we conclude that $\lim_{c \rightarrow +\infty} \breve{I}_{c}(\breve{\phi}_{c}) = J_{1}^{0,q}(\psi_{1,q})$.

    The rest of the proof is almost the same as in Lemma \ref{lem:suffcond2_min_seq;gkdv_st}.
\end{proof}

Using Lemma \ref{label:min_seq_caseII1;gkdv_st}, we can prove Lemma \ref{lem:convergence;gkdv_st} similarly to case (I).

Now we consider case (II-2), where $a = -1$, $p$ is even, and $q$ is odd.
We claim the following statement.

\begin{lem} \label{lem:suffcond2_min_seq_II2;gkdv_st}
    Assume condition {\rm (II-2)} in Theorem \ref{thm:existence_of_GS;gkdv_st}.
    Let $\phi_{c}$ be a negative ground state solution to \eqref{eq:Intro0300;gkdv_st} $c > 0$, and $\breve{\phi}_{c}$ is a function obtained by the scaling \eqref{eq:coercivity0300;gkdv_st}.
    Moreover, let $\psi_{1,q}$ be the positive ground state solution to \eqref{eq:coercivity0500;gkdv_st} with $r = q$, and $\chi_{1,q} = - \psi_{1,q}$.
    Then the followings hold:
    \begin{enumerate}[label={\rm (\roman*)} \ ]
        \item $\displaystyle \lim_{c \rightarrow \infty} J_{1}^{0, q}(\breve{\phi}_{c}) = J_{1}^{0, q}(\chi_{1, q})$;
        \item $\displaystyle \lim_{c \rightarrow \infty} K_{1}^{0,q}(\breve{\phi}_{c}) = 0$.
    \end{enumerate}
\end{lem}
\begin{proof}
    Now we put
    \begin{equation}
        I(v) \coloneq \left( \frac{1}{2} - \frac{1}{p+1} \right) \| v \|_{H^{\sigma/2}}^{2} + \left( \frac{1}{p+1} - \frac{1}{q+1} \right) \| v \|_{L^{q+1}}^{q+1},
    \end{equation}
    for $v \in H^{\sigma/2}(\mathbb{R})$.
    We can see that
    \begin{equation}
        I(v) = \breve{S}_{c}(v) - \frac{1}{p+1} \breve{K}_{c}(v) =S_{1}^{0,q}(v) - \frac{1}{p+1} K_{1}^{0,q}(v).
    \end{equation}
    Then, we can see the following characterizations:
    \begin{align}
        \breve{S}_{c}(\breve{\phi}_{c}) &= \inf \{ \breve{S}_{c}(v): v \in H^{\sigma/2}(\mathbb{R}) \setminus \{0\}, \ \breve{K}_{c}(v) = 0 \} \\
        &= \inf \{ I(v): v \in H^{\sigma/2}(\mathbb{R}) \setminus \{0\}, \ \breve{K}_{c}(v) = 0 \} \\
        &= \inf \{ I(v): v \in H^{\sigma/2}(\mathbb{R}) \setminus \{0\}, \ \breve{K}_{c}(v) \leq 0 \}, \label{eq:suffcond2_2700;gkdv_st} \\
        S_{1}^{0,q}(\chi_{1,q}) = S_{1}^{0,q}(\psi_{1,q}) &= \inf \{ S_{1}^{0,q}(v): v \in H^{\sigma/2}(\mathbb{R}), \ K_{1}^{0,q}(v) = 0 \} \\
        &= \inf \{ J_{1}^{0,q}(v): v \in H^{\sigma/2}(\mathbb{R}), \ K_{1}^{0,q}(v) = 0 \} \\
        &= \inf \{ J_{1}^{0,q}(v): v \in H^{\sigma/2}(\mathbb{R}), \ K_{1}^{0,q}(v) \leq 0 \} \\
        &= \inf \{ I(v): v \in H^{\sigma/2}(\mathbb{R}), \ K_{1}^{0,q}(v) = 0 \} \\
        &= \inf \{ I(v): v \in H^{\sigma/2}(\mathbb{R}), \ K_{1}^{0,q}(v) \leq 0 \} \label{eq:suffcond2_2710;gkdv_st}
    \end{align}

    Since $K_{1}^{0, q}(\chi_{1, q}) = 0$, we have
    \begin{equation}
        \breve{K}_{c}(\chi_{1, q}) = c^{-\beta} \int_{\mathbb{R}} \chi_{1, q}^{p + 1} \, dx = - c^{-\beta} \int_{\mathbb{R}} \psi_{1, q}^{p+1} \, dx < 0
    \end{equation}
    for $c > 0$.
    Combining this with \eqref{eq:suffcond2_2700;gkdv_st} and \eqref{eq:suffcond2_2710;gkdv_st}, we obtain
    \begin{equation}
        J_{1}^{0, q}(\breve{\phi}_{c}) \leq I(\breve{\phi}_{c}) \leq I(\psi_{1, q}) = J_{1}^{0, q}(\psi_{1, q}),
    \end{equation}
    which implies that $\limsup_{c \rightarrow +\infty} J_{1}^{0, q}(\breve{\phi}_{c}) \leq J_{1}^{0, q}(\psi_{1, q})$ and that the family $\{ \breve{\phi}_{c} \}_{c > 0}$ is bounded in $H^{\sigma/2}(\mathbb{R})$.

    The rest of the proof is almost the same as in Lemmas \ref{lem:suffcond2_Nehari_est;gkdv_st} and \ref{lem:suffcond2_min_seq;gkdv_st}.
\end{proof}

By Lemma \ref{label:min_seq_caseII1;gkdv_st}, we can prove Lemma \ref{lem:convergence;gkdv_st} for case (II-2) in a similar way to case (I).

\section*{Acknowledgment}
The author would like to thank Professor Masahito Ohta for his encouragements and helpful discussions for this study.

\bibliography{reference_list}
\bibliographystyle{amsplain}

\end{document}